\documentclass[10pt,a4paper]{scrartcl}
\usepackage[utf8x]{inputenc}
\usepackage{amsmath}
\usepackage{amsthm}
\usepackage{amssymb}
\usepackage{amsfonts}
\usepackage{bm}
\usepackage{graphicx}
\usepackage{color}
\usepackage{dsfont}
\usepackage{subcaption}
\usepackage{placeins}

\newtheorem{theorem}{Theorem}
\newtheorem{lemma}{Lemma}

\newtheorem{remark}{Remark}
\newcommand{\R}{{\mathbb R}}

\newcommand{\bk}[2]{\left\langle #1,#2\right\rangle}

\newcommand{\argmin}{\mathop{\rm argmin}}%

\newcommand{\cB}{\mathcal{B}}

\newcommand{\cS}{{\mathcal S}}

\newcommand{\cA}{{\mathcal A}}

\newcommand{\bE}{{\mathbf E}}
\newcommand{\bP}{{\mathbf P}}

\newcommand{\eps}{\epsilon}

\newcommand{\wrt}{with respect to }

\title{Meshless discretization of LQ-type stochastic control problems}
\author{Ralf Banisch \and Carsten Hartmann}

\begin{document}

\maketitle

\begin{abstract}
We propose a novel Galerkin discretization scheme for stochastic optimal control problems on an indefinite time horizon. The control problems are linear-quadratic in the controls, but possibly nonlinear in the state variables, and the discretization is based on the fact that problems of this kind can be transformed into linear boundary value problems by a logarithmic transformation. We show that the discretized linear problem is dual to a Markov decision problem, the precise form of which depends on the chosen Galerkin basis. We prove a strong error bound in $L^{2}$ for the general scheme and discuss two special cases: a variant of the known Markov chain approximation  obtained from  a basis of characteristic functions of a box discretization, and a sparse approximation that uses the basis of committor functions of metastable sets of the dynamics; the latter is particularly suited for high-dimensional systems, e.g., control problems in molecular dynamics. We illustrate the method with several numerical examples, one being the optimal control of Alanine dipeptide to its helical conformation.  
\end{abstract}

\tableofcontents

\section{Introduction}

A large body of research is concerned with the question: How well can a continuous diffusion in an energy landscape be approximated by a Markov jump process (MJP) in the regime of low temperatures? Qualitatively, the approximation should be good if the system under consideration is metastable, for metastability precisely means that the diffusion process stays in the neighbourhood of the potential energy minima for a long time, and occasionally makes rapid transitions (jumps) between the wells. These metastable regions then become the states of the MJP, and the jump rates are determined by the frequency of the transitions. A rigorous mathematical proof of this fact, based on Gamma convergence, has recently been published for the special case of a symmetric double-well potential and in the limit of zero temperature \cite{Peletier2012,Mielke2012}. From a more practical point of view, Markov state models (e.g.~see \cite{Pande2010,Prinz,Schuette2011,Djurdjevac2010}) that are popular in the molecular dynamics 
community are approximations of metastable systems by MJPs that, in certain cases, can capture the correct transition rates between the metastable sets  \cite{Sarich2010,Djurdjevac2012}. 

The situation is more complicated for controlled diffusions, when the processes are subject to external controls that are chosen so as to minimize a given cost criterion, in which case one has to approximate the stochastic processes (in an appropriate sense) as well as the corresponding cost functional and the resulting optimal control forces. Available numerical methods for solving stochastic optimal control problems include methods that solve the dynamic programming PDE or Hamilton-Jacobi-Bellman equation, such as Markov chain approximations \cite{Kushner1992}, monotone schemes \cite{barles1991,barles2007} or finite elements \cite{Boulbrachene2001}. A common problem for these PDE-based methods is that they become inefficient if the problems are high-dimensional. An alternative are direct methods that iteratively minimize the cost functional using, e.g, entropy minimization \cite{todorov2009}, path integrals \cite{kappen2005} or policy iteration \cite{Howard2009}. This class of methods works in high 
dimensions, but has difficulties if the solvers get stuck in local minima or if the search space is too large. 

In this article we propose a Galerkin scheme for discretizing the dynamic programming equations that is meshless and hence can be applied to large-scale problems, assuming that the Galerkin basis is chosen in a clever way. Galerkin or, more generally, reduced basis methods are well established in terms of theory and numerical algorithms for linear elliptic equations, but to our knowledge very few results (e.g.~\cite{Douglas1970,Schuette2012}) are available for the nonlinear dynamic programming equations of stochastic control; cf.~also \cite{Kunisch2001}. To close this gap we confine ourselves to a class of optimal control problem that are linear-quadratic (LQ) in the control variables, but possibly nonlinear in the states; they have the feature that they can be transformed to a linear problem by a suitable (logarithmic) transformation and are hence amenable to a discretization using Galerkin methods. LQ-type optimal control problems of this kind appear relevant in many applications, including molecular dynamics \cite{Schuette2012,Stapelfeldt2004}, 
material science \cite{Steinbrecher2010,Asplund2011}, quantum computing \cite{Kosloff2002,rabitz2000}, or queuing networks \cite{asmussen1995,heidelberger1995} to mention just a few; see also \cite{WangDupuis2004,is_multiscale} for an application in statistics.

\paragraph{A simple paradigm}

As an introductory example consider the one-dimensional diffusion process $(X_t)_{t\ge 0}$ satisfying the It\^o stochastic differential equation 
\begin{equation*}
 dX_t = b(X_t,t)dt + \sqrt{2\epsilon} dB_t\,,t\ge 0
\end{equation*}
where $B_{t}$ is standard one-dimensional Brownian motion, $\eps>0$ is noise intensity, called \emph{temperature} in the following, and $b(\cdot,\cdot)$ is a smooth and Lipschitz continuous vector field. Specifically, we assume that $b$ is of the form
\begin{equation*}
b(x,t) = u_{t} - \nabla V(x)\,,
\end{equation*}
with $V\colon\R\to\R$ being a smooth potential energy function that is bounded from below, and measurable control $u$, that will be chosen so as to minimize a certain cost criterion. As an example consider the situation depicted in Figure \ref{fig:idea} and suppose that the control task is to force the particle in the left well to reach the right well in minimum time. For $u=0$ and in the limit of small noise, the average of the transition time from the left to the right well is given by the Kramers-Arrhenius law \cite{freidlin1984,berglund2011} 
\begin{equation*}
\bE[\tau] \asymp \exp(\Delta V/\eps)\,,\quad\eps\to 0\,,
\end{equation*}
where $\tau$ is the first exit time from the left well, $\Delta V$ is the energy barrier between the wells and $\bE[\cdot]$ is the expectation over all realizations of the process. We can speed up the transitions by tilting the potential according to $V(x)\mapsto V(x)-ux$, making the barrier smaller. If were allowed to apply arbitrarily large forces, we could make $\tau$ arbitrarily small, yet a real controller will seek to minimize the transition time without controlling too much. A natural choice in many cases is a penalization of the energy consumed by the controller, which leads to quadratic cost of the form \cite{Schuette2012}
\begin{equation*}
J(u) = \bE\left[\tau + \frac{1}{4}\int_{0}^{\tau}|u_{t}|^{2}dt\right]\,.
\end{equation*}
(The factor $1/4$ is for convenience.) The showcase optimal control problem now reads: 
\begin{equation}\label{minJ}
\min_{u\in\cA} J(u)
\end{equation}
over a set $\cA$ of admissible (e.g.~adapted) strategies and subject to 
\begin{equation}\label{SDE0}
 dX^{u}_t = (u_{t} - \nabla V(X^{u}_{t}))dt + \sqrt{2\epsilon} dB_t\,.
\end{equation}

\begin{figure}
 \centering
    \includegraphics[width=0.48\textwidth]{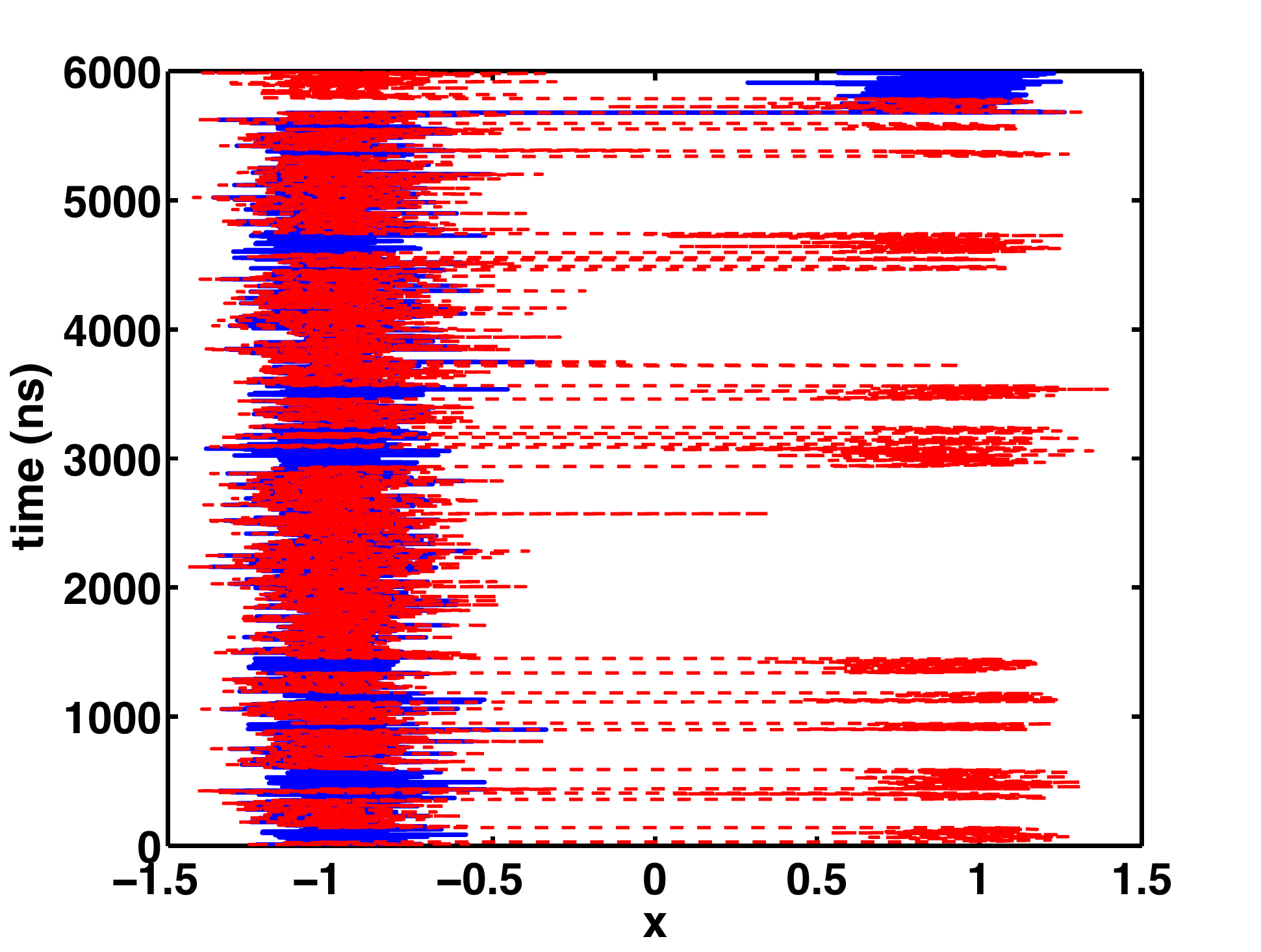} 
 \includegraphics[width=0.48\textwidth]{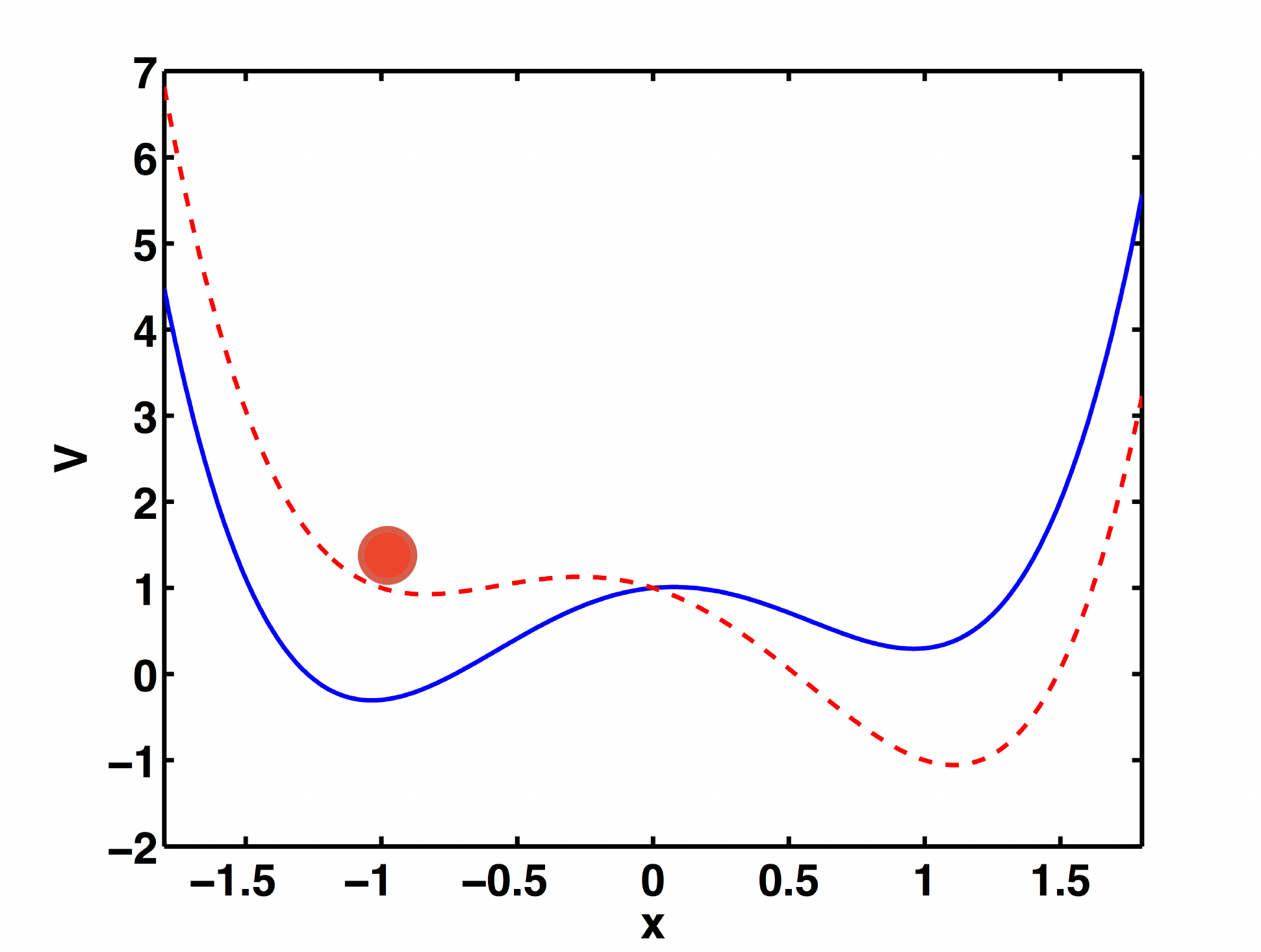}  
 \caption{Two typical realization of the bistable system, with and without tilting (left panel). The corresponding potential energies are shown in the right panel.}
 \label{fig:idea}
\end{figure}
In this paper we deal with the question how to solve optimal control problems of the form (\ref{minJ})--(\ref{SDE0}), beyond simple one-dimensional examples. The typical application that we have in mind is molecular dynamics that is high-dimensional and displays vastly different time scales. This defines the basic requirements of the numerical method used to solve optimal control problems: it must handle problems with large state space dimension and it must be able to capture the relevant processes of the dynamics, typically the slowest degrees of freedom in the system.         

For moderate controls, and if the temperature is small compared to the energy barrier, the dynamics in the above example basically consists of rare jumps between the potential wells, with the jump rate being controlled by $u$.  Therefore an efficient discretization would be one that resolves only the jumps between the minima by a 2-state Markov jump process with adjustable jump rates, according to the value of the control variable. If the control $u_{t}$ is given, the approximation of (\ref{SDE0}) by a 2-state MJP essentially boils down to an approximation of the dynamics by a time-inhomogeneous MSM, which requires only minor generalizations of the homogeneous MSM framework (see, e.g., \cite{Djurdjevac2012,Sarich2010}). When solving optimal control problems, however, the control becomes a function of the dynamics, for (\ref{SDE0}) enters as a constraint in the minimization problem (\ref{minJ}), which makes the corresponding dynamic programming equations nonlinear and renders the discretization less straightforward. The discretization scheme that we propose is based on the fact that the above control problem can be transformed into linear boundary value problems by a logarithmic transformation. The linear problem can then be discretized by standard means, including the discretization by an MSM if the dynamics is metastable. As we will show below the discretized linear problem is dual to a Markov decision problem (i.e.~a stochastic control problem for a continuous-time finite-state Markov process), which thus represents the natural Markovian discretization of the original stochastic control problem. The discretization is \emph{meshless}, in that the number of states of the Markov model does not scale exponentially with the dimension of the continuous state space, hence our method avoids the curse of dimensionality of typical grid based methods.

\paragraph{Organization of the article}
The rest of the paper is organised as follows: In section 2 we introduce the class of optimal control problems studied and state the duality between optimal control and sampling for both continuous SDEs and MJPs. In section 3, the Galerkin projection method is introduced, and some results about the approximation error are discussed. We also give a stochastic interpretation of the discretized linear equation in terms of Elber's milestoning process \cite{elber2004} and discuss special cases of the discretization, one of which being the known Markov chain approximation. Finally, we construct sampling estimators. Section 3 is the core parts of the paper and contains new results, including a strong $L^{2}$ error estimate for the Galerkin discretization. In section 4, we discuss numerical examples.

\subsection{Elementary notation and assumptions}

We implement the following notation and standing assumptions that will be used throughout the paper and that generalize the above example. Our optimal control problem has the following ingredients: 

\paragraph{Dynamics} 

Let $\cS\subset\R^{d}$ bounded with smooth (e.g.~Lipschitz) boundary and consider the potential energy function $V\colon\cS\to\R$, that we assume to be two times continuously differentiable and bounded from below. We consider $X_{t}\in\cS$ solving 
\begin{equation}\label{SDE}
 dX^{u}_t = (u_{t} - \nabla V(X^{u}_{t}))dt + \sqrt{2\epsilon} dB_t\,, \quad t\ge 0\,,
\end{equation}
where $B_{t}\in\R^{d}$ is $d$-dimensional Brownian motion under a probability measure $P$, and $u\colon[0,\infty)\to U\subset\R^{d}$ is a time dependent measurable and bounded function. We further impose reflecting boundary conditions at the set boundary $\partial\cS$, so that the process cannot leave the set $\cS\subset\R^{d}$; see \cite{gardiner1996} for algorithmic issues. 

\paragraph{Reversibility and invariant measure} 
For test functions $\varphi\colon\cS\to\R$ that are two times continuously differentiable, the infinitesimal generator of the uncontrolled process $X_{t} = X_{t}^{0}$ is defined as the second-order differential operator 
\[
L\varphi = \eps\Delta\varphi  - \nabla V\cdot\nabla\varphi\,.
\]
Define 
\[
d\mu(x) = \exp(-\eps^{-1}V(x)) dx\,
\] 
to be the Boltzmann measure at temperature $\eps>0$. Without loss of generality, we assume that $\mu$ is normalized, so that $\mu(\cS)=1$. 
For the subsequent analysis it will be convenient to think of $L$ as an operator acting on a suitable subspace of 
\[
L^{2}(\cS,\mu) =\left\{\phi\colon\cS\to\R \colon \int_{\cS} |\phi(x)|^{2}\,d\mu(x)<\infty \right\}\,,
\]
that is a weighted Hilbert space equipped with the scalar product
\[
\bk{v}{w}_{\mu} = \int_{\cS} v(x) w(x)d\mu(x)\,.
\]
It can be readily seen that $L$ is symmetric \wrt the weighted scalar product, \[
\bk{L v}{w}_{\mu}=\bk{v}{L w}_{\mu}\,,
\]
which implies that $X_{t}$ is reversible \wrt  the Boltzmann measure $\mu$. Moreover $\mu$ is the unique invariant measure of the process $X_{t}$ and satisfies
\[
\int_{\cS}(L\psi)d\mu =  \int_{\cS}\psi(L \mathds{1})d\mu = 0
\]
for all test functions $\psi\in L^{2}(\cS,\mu)$.

\paragraph{Quadratic cost criterion} 

We now introduce the cost criterion that the controller choosing $u$ in (\ref{SDE}) seeks to minimize. To this end let $A\subset\cS$ be a bounded subset that is fully contained in the interior of $\cS$ and call $\tau_A<\infty$ the random stopping time 
\[
\tau_A = \inf\{ t>0\colon X_{t}\in A\}\,.
\]
We define the cost functional 
\begin{equation}\label{J}
J(u) = \bE\left[\int_{0}^{\tau_A}\left\{f(X_{t}) + \frac{1}{4}|u_{t}|^{2}\right\}dt \right]\,,
\end{equation}
where $f\colon\cS\to\R$, called \emph{running cost}, is any nonnegative function with bounded first derivative; the factor $1/4$ in the penalization term is merely conventional. Cost functionals of this form are called \emph{indefinite time horizon cost}, because the terminal time $\tau_A$ is random. We will sometimes need the conditioned variant of the cost function, 
\begin{equation}\label{Jx}
J(u;x) = \bE_{x}\left[\int_{0}^{\tau_A}\left\{f(X_{t}) + \frac{1}{4}|u_{t}|^{2}\right\}dt \right], 
\end{equation}
where $\bE_{x}[\cdot]=\bE[\cdot|X_{0}=x]$ is a shorthand for the expectation over all realizations of $X_{t}$ starting at $X_{0}=x$, i.e.~the expectation \wrt $P$ conditional on $X_{0}=x$.

\paragraph{Admissible control strategies}

We call a control strategy $u=(u_{t})_{t\ge 0}$ \emph{admissible} if it is adapted to the filtration generated by $B_{t}$, i.e.,~if $u_{t}$ depends only on the history of the Brownian motion up to up to time $t$, and if the equation for $X^{u}_{t}$ has a unique strong solution. The set of admissible strategies is denoted by $\cA$. 

Even though $u_{t}$ may depend on the entire past history of the process up to time $t$, it turns out that optimal strategies are Markovian, i.e., they depend only on the current state of the system at time $t$. In our case, in which the costs are accumulated up to a random stopping time $\tau_A$, the optimal strategies are of the form
\[
u_{t} = \alpha(X^{u}_{t})
\]  
for some function $\alpha\colon\R^{d}\to\R^{d}$. Hence the optimal controls are time-homogeneous feedback policies, depending only on the current state $X_{t}^{u}$, but not on $t$.

\section{Optimal control and logarithmic transformation}\label{sec2}

In this section we establish a connection between optimal control of continuous-time Markov processes and certain path sampling problems, the latter are associated with a linear boundary value partial differential equation (PDE) that can be discretized by standard numerical techniques for PDEs or Monte-Carlo. The duality between optimal control and path sampling goes back to Wendel Fleming and co-workers (e.g.~\cite{fleming1977,fleming1995,james1992}) and is based on a logarithmic transformation of the function
\begin{equation}\label{valuefct}
W(x) = \min_{u\in\cA} J(u;x)\,.
\end{equation}

\subsection{Duality between control and path sampling for diffusions}\label{ssec:logtrafo} 
 
Our simple derivation of the duality between path sampling optimal control will be based on the Hamilton-Jacobi-Bellman equations of optimal control. To this end, we recall the dynamic programming principle for optimal control problems of the form (\ref{SDE})--(\ref{J}) that we state without proof; for details we refer to, e.g., \cite[Secs.~VI.2]{fleming2006}. 
\begin{theorem}
Let $W\in C^{2}(\cS)\cap C^{1}(\bar{\cS})\cap C(\bar{A})$ be the solution of
\begin{equation}\label{hjb}
\begin{aligned}
\min_{c\in\R^{d}}\left\{ (L + c\cdot\nabla)W(x) + f + \frac{1}{4}|c|^{2} \right\} = 0\,, &\quad x\in\cS\setminus A \\ 
W(x) = 0\,, & \quad x\in A\\ \nu\cdot\nabla W(x)= 0\,, & \quad x\in \partial\cS\,,
\end{aligned}
\end{equation}
where $\nu$ is the outward-pointing unit normal to $\partial\cS$ at $x$. Then
\[
W(x) = \min_{u\in\cA} J(u;x)\,
\]
where the minimizer $u^{*}=\argmin J(u)$ is unique and given by the feedback law 
\begin{equation}\label{oc}
u_{t} = -2\nabla W(X_{t}^{u})\,.
\end{equation}
\end{theorem}

Before we proceed with the derivation of the dual sampling problem, we shall briefly discuss some of the consequences of the dynamic programming approach.  
Equation (\ref{hjb}) is the Hamilton-Jacobi-Bellman (HJB) equation, also called \emph{dynamic programming equation} associated with the following optimal control task:
\[
\min_{u\in\cA} J(u) \quad\textrm{s.t.}\quad  dX^{u}_t = (u_{t} - \nabla V(X^{u}_{t}))dt + \sqrt{2\epsilon} dB_t\,.
\]

The function $W(x)$ is called \emph{value function} or \emph{optimal cost-to-go}. Existence and uniqueness of classical (i.e.~smooth) solutions follow from our assumptions on the potential and the properties of $\cS$, using the results in \cite[Secs.~VI.3--5]{fleming2006}. 

Given the value function, and using the fact that optimal control is the gradient of two times the value function, the optimally controlled process $X^{*}_{t}$ solves the SDE
\begin{equation}\label{SDE_controlled}
 dX^{*}_t = - \nabla U(X^{*}_{t}))dt + \sqrt{2\epsilon} dB_t\,. 
\end{equation}
with the new potential 
\[
U(x)=V(x)+2W(x)\,.
\]
Note that $X^{*}_{t}$ is reversible \wrt a tilted Boltzmann distribution having the density $\rho^{*}=\exp(-U/\eps)$. The reversibility follows from the fact that the value function does not depend on $t$, which would not be the case if the terminal time $\tau_A$ were a deterministic stopping time rather than a first exit time.\footnote{For finite time-horizon control problems the value function depends on the time $\tau_A-t$ remaining until the terminal time $\tau_A$.}

\paragraph{Logarithmic transformation and Feynman-Kac formula (part I)}
The approach that is pursued in this article is to discretize the HJB equation by first removing the nonlinearity by a logarithmic transformation of the value function. Let
\begin{equation}\label{logtrafo1}
\phi(x) = \exp(-\eps^{-1}W(x))\,.
\end{equation}  
It follows by chain rule that   
\begin{equation}\label{logtrafo2}
\eps \frac{L \phi}{\phi}= -L W + |\nabla W|^{2}\,,\quad \phi\neq 0\,,
\end{equation} 
which, together with the relation    
\begin{equation*}
-|\nabla W|^{2} = \min_{c\in\R^{m}}\left\{c\cdot\nabla W + \frac{1}{4}|c|^{2}\right\}\,,
\end{equation*}
implies that (\ref{hjb}) is equivalent to the linear boundary value problem
\begin{equation}\label{linbvp}
\begin{aligned}
\left(L - \eps^{-1}f\right) \phi(x)  = 0\,, & \quad x\in \cS\setminus A \\ \phi(x) = 1\,, & \quad x\in A\\ 
\nu \cdot\nabla\phi(x)= 0\,, & \quad x\in\partial\cS\,.
\end{aligned}
\end{equation}
By the above assumptions and the strong maximum principle for elliptic PDEs it follows that (\ref{linbvp}) has a classical solution $\phi^{\eps}\in C^{2}(\cS)\cap C(\bar{\cS})\cap C(\bar{A})$ that is uniformly bounded. Further note that the value function is uniformly bounded on $\cS$, hence the log transformation (\ref{logtrafo1})--(\ref{logtrafo2}) is well defined. 

Now, by the Feynman-Kac theorem \cite[Thm.~8.2.1]{oksendal2003}, the linear boundary value problem has an interpretation in terms of a sampling problem. The solution (\ref{linbvp}) can be expressed as the conditional expectation   
\begin{equation}\label{phi}
\phi(x) = \bE_{x}\left[\exp\left(-\frac{1}{\eps}\int_{0}^{\tau_A} f(X_{t})\,ds\right)\right]
\end{equation}
over all realizations of the following SDE on $\cS$: 
 \begin{equation}\label{sde_sans}
 dX_t = - \nabla V(X_{t})dt + \sqrt{2\epsilon} dB_t\,, \quad X_{0}=x\,.
\end{equation}
\begin{remark}
The Neumann boundary condition in (\ref{hjb}) and (\ref{linbvp}) amounts to the reflecting boundary conditions for the processes $X^{u}_{t}$ and $X_{t}$ at $\partial\cS$. 
\end{remark}

\begin{remark}
In probabilistic terms, the logarithmic transformation amounts to a suitable change of measure of the underlying Markov process, by which the control variable is eliminated \cite{Pra1996}; see also \cite{Hartmann2012} and the references given there. 
\end{remark}

\subsection{Duality between control and path sampling for jump processes}\label{sec_MJP}

In the last section, we have established a connection between an optimal control problem and sampling of a continuous path observables $\phi(x)$. In this section, we will repeat the same construction for Markov jump processes, however, in reverse order: starting from a path observable for a Markov jump process, we derive the dual optimal control using a logarithmic transformation. 

Let $(\hat X_t)_{t\geq 0}$ be a MJP on the discrete state space $\hat{\cS} = \{1,\ldots,n\}$ with infinitesimal generator $G \in\R^{n\times n}$. 
The entries of the generator matrix $G$ satisfy
\[
G_{ij}\ge 0 \;\textrm{ for $i\neq j$ and }\; G_{ii} = -\sum_{j\neq i} G_{ij}\,,
\]
where the off-diagonal entries of $G$ are the jump rates between the states $i$ and $j$. 
\paragraph{Logarithmic transformation and Feynman-Kac formula (part II)}
In accordance with the previous subsection let $\hat{f}: \hat{\mathcal{S}} \to\R$ be nonnegative and define the stopping time 
\[
\tau_A = \inf\{ t> 0\colon \hat X_t \in A\}.
\]
to be the first hitting time of a subset $A\subset\hat{\mathcal{S}}$. As before we introduce a function 
\[
\hat{\phi}(i) = \bE_{i}\left[\exp\left(-\frac{1}{\eps}\int_0^{\tau_A}\hat{f}(\hat{X}_s)ds\right)\right],
\]
with $\bE_{i}[\cdot]=\bE[\cdot| \hat{X}_0 = i]$ being the conditional expectation over the realizations of $\hat{X}_{t}$ starting at $\hat{X}_{0}=i$. We have the following Lemma that is the exact analogue of the Feynman-Kac formula for diffusions for the case of an MJP (see  \cite{Skorokhod1975}). 
\begin{lemma}\label{lemma_backw_mjp}
The function $\hat{\phi}(i)$ solves the linear boundary value problem  
\begin{equation} \label{MJP_FK}
\begin{aligned}
 \sum_{j\in \hat{\mathcal{S}}} G_{ij} \hat{\phi}(i) - \eps^{-1} \hat{f}(i) \hat{\phi}(i) = 0\,, & \quad i\in \hat{\cS}\setminus A \\
 \hat{\phi}(i) = 1\,, & \quad i\in A\,.
\end{aligned}
\end{equation}
\end{lemma}

Now, in one-to-one correspondence with the log transformation procedure in the diffusion case, the function
\[
\hat W = -\eps\log \hat \phi
\] 
can be interpreted as the value function of an optimal control problem for the MJP $(\hat X_t)_{t\geq 0}$. The derivation of the dual optimal control problem goes back to \cite{SheuPhD,Sheu1985}, and we repeat it here in condensed form for the reader's convenience (see also \cite[Sec.~VI.9]{fleming2006}):  
First of all note that $\hat W$ satisfies the equation
\begin{align*}
\exp(\hat W/\eps) G\exp(-\hat W/\eps) - \eps^{-1}\hat{f} = 0\,, & \quad i\in A^c\\ \hat W(i) = 0\,, & \quad i\in A\,.
\end{align*}
and define a new generator matrix by 
\[
G^{v}=(G^{v}_{ij})_{i,j\in\hat{\cS}} \,, \quad G^v_{ij} = \frac{G_{ij}v(j)}{v(i)}\,,
\]
with $v(i)>0$ for all $i\in\hat{\cS}$. The exponential term in above equation for $\hat{W}$ can be recast as
\[
\frac{(G\hat{\phi})(i)}{\hat{\phi}(i)} = \min_{v>0}\{ - (G^{v}\hat{W})(i) + k^{v}(i) \}
\]
where we have introduced the shorthand 
\[
k^v(i) = \eps(G^v(\log v))(i) - \eps\frac{(Gv)(i)}{v(i)}\,,
\] 
and used the identity
\[
\min_{y\in\R}\{e^{-y} + ay\} = a− a\log a\,,\quad a>0\,.
\]
As a consequence, (\ref{MJP_FK}) is equivalent (i.e.~dual) to 
\begin{equation}\label{MJP_HJB}
\begin{aligned}
\min_{v>0} \left\{(G^v \hat W)(i) + k^v(i) + \hat{f}(i)\right\} = 0\,, & \quad  i\in A^c\\  
\hat W(i) = 0\,, &  \quad i\in A\,.
\end{aligned}
\end{equation}
which is the dynamic programming equation of a Markov decision problem, i.e.~an optimal control problem for an MJP (e.g.~see \cite[Sec.~VI.9]{fleming2006}): Minimize
\begin{equation} \label{MJP_cost}
 \hat J(v) = \bE\left[\int_0^{\tau_A}\left\{\hat{f} (\hat X_s^v) + k^v(\hat X_s^v)\right\}ds\right]
\end{equation}
over all component-wise strictly positive controls $v$ and subject to the constraint that the process $(\hat X_t^v)_{t\geq 0}$ is generated by $G^v$. It readily follows from the derivation of (\ref{MJP_HJB}) that the minimizer exists and is given by 
\[
v^*(i) = \hat\phi(i)\,.
\]  
The next lemma records some important properties of the controlled Markov jump process with generator $G^{v}$ and the corresponding cost functional (\ref{MJP_cost}).

\begin{lemma}\label{lemma_props}
Let $G^v$ and $k^v$ be  defined as above. 
\begin{enumerate}
\item[(i)] Let $(\hat{X}_{t})_{t\ge 0}$ with generator $G$ have a unique stationary distribution $\pi$ and let $G$ be reversible with respect to $\pi$. Then the tilted distribution $\pi^v(i) = Z_v^{-1} v^2(i)\pi(i)$, with $Z_v$ an appropriate normalization constant, is the unique probability distribution such that $G^v$ is reversible and stationary with respect to $\pi^v$.
\item[(ii)] Let $\hat P$ denote the probability measure on the space of trajectories generated by $\hat{X}_t$ with initial condition $\hat{X}_{0}=i$, and let $\hat Q$ be the corresponding probability measure generated by $\hat{X}^{v}_{t}$ with the same initial condition $\hat{X}^{v}_{0}=i$. Then $\hat Q$ is absolutely continuous \wrt $\hat P$ and the expected value of the running cost $k^{v}$ is the Kullback-Leibler (KL) divergence between $\hat Q$ and $\hat P$, i.e., 
\[
\bE_{\hat Q}\left[\int_0^{\tau_A} k^{v}(\hat X^{v}_s)ds\right] = \int \log\frac{d\hat Q}{d\hat P} d\hat Q
\]
where $\bE_{\hat Q}[\ldots]$ denotes expectation over all realizations of $\hat{X}^{v}_{t}$ starting at $\hat{X}^{v}_{0}=i$. 

\end{enumerate}
\end{lemma}

\begin{proof}
We first show $(i)$. By assumption we have $\pi(i)G_{ij} = \pi(j)G_{ji}$. Now, let $\pi^v$ be such that $\pi^v(i)G^v_{ij} = \pi^v(j)G^v_{ji}$. We will show that $\pi^v$ has the proposed form:
$$\pi^v(i)G^v_{ij} = \frac{v(j)}{v(i)}  \frac{\pi^v(i)}{\pi(i)}\pi(i)G_{ij} =  \frac{v(j)}{v(i)} \frac{\pi^v(i)}{\pi(i)}\pi(j)G_{ji} =  \frac{v^2(j)}{v^2(i)} \frac{\pi(j)}{\pi(i)}\frac{\pi^v(i)}{\pi^v(j)} \pi^v(j)G^v_{ji}$$
But since $\pi^v(i)G^v_{ij} = \pi^v(j)G^v_{ji}$, we must have
$$\frac{\pi^v(j)}{\pi(j)v^2(j)} = \frac{\pi^v(i)}{\pi(i)v^2(i)} \quad\forall i\neq j.$$
This can only be true if the quantity $Z_v^{-1} =  \frac{\pi^v(i)}{\pi(i)v^2(i)}$ is independent of $i$. This gives $\pi^v(j) = Z_v^{-1} v^2(j)\pi(j)$ as desired. The constant $Z_v$ is uniquely determined by the requirement that $\pi^v$ be normalized. Finally, from reversibility it follows directly that $\pi^v$ is also a stationary distribution of $G^v$.

To show $(ii)$, note that the running cost $k^v(i)$ can be written as
\begin{equation} \label{kv}
 k^v(i) = \eps\sum_{j\neq i} G_{ij}\left\{ \frac{v(j)}{v(i)}\left[\log \frac{v(j)}{v(i)} - 1\right] + 1\right\},
\end{equation}
which is the KL divergence between $\hat Q$ and $\hat P$ (see \cite[Sec.~3.1.4]{Pra1996}). The absolute continuity between $\hat Q$ and $\hat P$ simply follows from the fact that $v$ in the definition of $G^{v}$ was required to be component-wise strictly positive.     
\end{proof}

\begin{remark}\label{rem_props}

To reveal further similarities between the stochastic control problem (\ref{SDE})--(\ref{J})  and the corresponding Markov decision problem, note that the quadratic penalization term in (\ref{J}) equals the KL divergence between the reference measure $P$ of the uncontrolled diffusion  (\ref{sde_sans}) and the corresponding probability measure $Q$ that is induced by replacing $B_{t}$ in (\ref{sde_sans}) by 
\[
B^{u}_{t} = B_{t} + \sqrt{\frac{1}{2\eps}}\int_{0}^{t} u_{s}dt\,,
\]     
as can be shown using Girsanov's theorem \cite[Thm.~8.6.8]{oksendal2003}. In other words, it holds that (cf.~\cite{Hartmann2012,HaEtal13a})
\[
\bE_Q\left[\frac{1}{4}\int_0^{\tau_A} |u_{s}|^{2}ds\right]  = \eps\int \log\frac{dQ}{dP} dQ
\]
\end{remark}

\section{Discretization: Galerkin projection point of view}

In this section we will develop a discretization for optimal control problems of the type discussed in Section \ref{sec2}. The discretization will approximate the continuous control problem with a control problem for a Markov jump process on finite state space.

Because of the nonlinearity of the problem, a general theory for discretizing continuous optimal control problems is unavailable. However, we saw in Section \ref{sec2} that for the control problems we are interested in, a logarithmic transform to a linear PDE is available. For linear PDEs, discretization theory in terms of Galerkin projections onto finite-dimensional subspaces of the PDE solution space exists. Our strategy will therefore be the one indicated in Figure \ref{diagram}.

\begin{figure}[ht]
 \centering
 \includegraphics[width=0.7\textwidth]{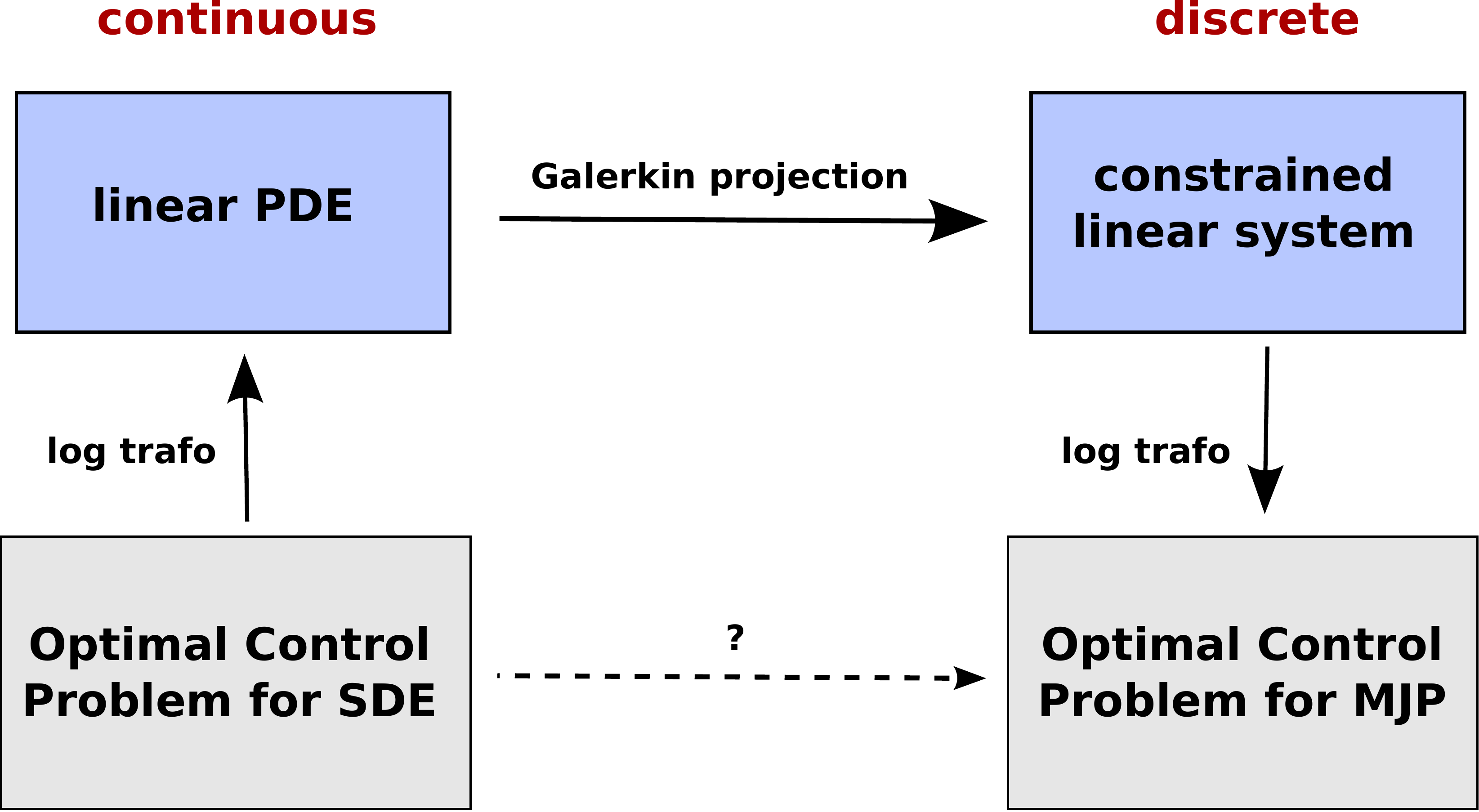}
 \caption{Discretization of continuous control problems via a logarithmic transform.}
 \label{diagram}
\end{figure}

In the first part of this section, we will develop the Galerkin projection for general subspaces and obtain some control of the discretization error. To refine this control, we specify the subspace $D$ we project onto. Specifically, we develop two possible choices for $D$ inspired by MSMs. The first choice specifies $D$ as the space of step functions on a full partition of state space, and if that full partition is chosen as a lattice with spacing $h$, then, as expected, the discretization error vanishes for $h\rightarrow 0$. The second choice uses a core partition of state space, where the cores are the metastable regions of the uncontrolled dynamics. We will prove a novel error bound which gives us detailed control over the discretization error even if very few basis functions are used.

In the second part of this section, we will develop the stochastic interpretation of the resulting matrix equation as the backward Kolmogorov equation of a MJP, which enables us to identify the discrete control problem for the MJP, as it was developed in Section \ref{sec2}. We will study the resulting discrete control problem for the two choices of $D$ specified earlier. In the full partition case, we will establish a connection to the finite volume approximation discussed in \cite{Latorre2011}, and we will show that to first order in $h$, our discrete control problem coincides with the Markov chain approximation constructed by Kushner \cite{Kushner1992}, confirming that the control problem itself converges to the continuous problem for $h\rightarrow 0$. In the incomplete partition case, we will make a connection to Transition Path Theory \cite{Vanden-Eijnden2006} and core set MSMs \cite{Schuette2011}.

\subsection{Galerkin projection of the Dirichlet problem}\label{sec_Galerkin}

As discussed above, we consider the boundary value problem
\begin{equation}\label{eq_bvp1}
\begin{aligned}
\left(L - \eps^{-1}f\right) \phi(x)  = 0\,, & \quad x\in A^{c}\\ \phi(x) = 1\,, & \quad x\in A\\ 
\nu \cdot\nabla\phi(x)= 0\,, & \quad x\in\partial\cS\,,
\end{aligned}
\end{equation}
with $L$ and $f$ as given above and $A^c = \cS\setminus A$. Following standard references we construct a Galerkin projection of (\ref{eq_bvp1}), see e.g. \cite{braack2012}. For this purpose, we introduce the $H^1$-norm $\|\phi\|_{H^1} =  \|\nabla u\|^2_\mu + \|u\|^2_\mu$, the Hilbert space $V = \{\psi\in L^2(\cS,\mu), \|\psi\|_{H^1} <\infty\}$ and the symmetric and positive bilinear form

$$\mathcal{B}:V\times V \rightarrow \R, \quad \mathcal{B}(\phi,\psi) = \eps^{-1}\langle f \phi, \psi\rangle_{\mu} + \eps\langle \nabla \phi, \nabla \psi\rangle_{\mu}.$$

Now if $\phi$ is a solution of (\ref{eq_bvp1}), then it also solves the weak problem

\begin{equation}
\cB(\phi,\psi) = 0 \quad \forall \psi\in V\,.
\label{eq_bvp2}
\end{equation}

A {\em Galerkin solution} $\hat\phi$ is any function satisfying

\begin{equation}
\cB(\hat\phi,\hat\psi) = 0 \quad \forall \hat\psi\in D,
\label{eq_bvp2a}
\end{equation}

with a predefined finite dimensional subspace $D\subset V$ that is adapted to the boundary conditions. In particular, we may choose basis functions $\chi_1,\ldots, \chi_{n+1}$ with the following properties:

\begin{itemize}
 \item[(S1)] The $\chi_i$ form a partition of unity, that is $\sum_{i=1}^{n+1}\chi_i = \mathds{1}$.
 \item[(S2)] The $\chi_i$ are adapted to the boundary conditions in (\ref{eq_bvp1}), that is $\nu\cdot\nabla\chi_i|_{\partial\cS} = 0$, $\chi_{n+1}|_A = 1$ and $\chi_i|_A = 0$ for $i\in \{1,\ldots,n\}$.
\end{itemize}

Then all elements of $D:= \chi_{n+1} \oplus D_0$ with $D_0 = \mbox{lin}\{\chi_1,\ldots,\chi_n\}$ will satisfy the boundary conditions in \ref{eq_bvp1}. Now define the matrices
\[
F_{ij} = \frac{\langle\chi_i,f \chi_j\rangle_\mu}{\langle\chi_i, \mathds{1}\rangle_\mu}, \qquad K_{ij} = \frac{\langle\chi_i,L\chi_j\rangle_\mu}{\langle\chi_i, \mathds{1}\rangle_\mu}\,.
\]
Then (\ref{eq_bvp2a}) takes the form of a matrix equation for the coefficients $\hat\phi = \sum_i \hat\phi_i\chi_i$:

\begin{equation}\label{eq_FK_disc}
\begin{aligned}
\sum_{j=1}^{n+1} \left(K_{ij} - \eps^{-1} F_{ij}\right) \hat{\phi}_j  & = 0\,, \quad i\in \{1,\ldots,n\}\\
\hat{\phi}_{n+1} & = 1\,,
\end{aligned}
\end{equation}

which is the discretization of (\ref{eq_bvp1}). 

\paragraph{Discretization error}

In order control the discretization error of the Galerkin method, we choose a norm $\|\cdot\|$ on $V$ and introduce the two error measures:

\begin{enumerate}
\item The {\em Galerkin error} $\bm{\varepsilon} = \|\phi-\hat\phi\|$, i.e. the difference between original and Galerkin solution measured in $\|\cdot \|$.
\item The {\em best approximation error} $\bm{\varepsilon_0} = \inf_{\hat\psi\in D} \|\phi-\hat\psi\|$, i.e. the minimal difference between the solution $\phi$ and any element $\hat\psi \in D$.
\end{enumerate}

In order to obtain full control over the discretization error, we need to obtain bounds on $\bm{\varepsilon}$, and we will do so by first obtaining a bound on the performance $\bm{p} := \bm{\varepsilon}/\bm{\varepsilon_0}$ and then a bound on $\bm{\varepsilon_0}$. The latter will depend on the choice of subspace $D$. For the former, standard estimates assume the following $\|\cdot\|$-dependent properties of $A$:

\begin{enumerate}
\item[(i)] Boundedness: $\cB(\phi,\psi) \leq \alpha_1 \|\phi\| \|\psi\|$ for some $\alpha_1 > 0$
\item[(ii)] Ellipticity: for all $\phi\in V$ holds $\cB(\phi,\phi) \geq \alpha_2 \|\phi\|$ for some $\alpha_2 > 0$.
\end{enumerate}

If both (i) and (ii) hold, Céa's lemma states that $\bm{p} \leq \frac{\alpha_1}{\alpha_2}$, see e.g. \cite{braack2012}. For the {\em energy norm} $\|\phi\|^ 2_B := \cB(\phi,\phi)$ we have $\alpha_1 = \alpha_2 = 1$ and therefore $\bm{p}=1$, thus the Galerkin solution $\hat \phi$ is the best-approximation to $\phi$ in the energy norm.

The next two Lemmas give a bound on $\bm{p}$ if errors are measured in the $L^2$-norm $\|\cdot\|_\mu$. In this case, $\cB(\cdot,\cdot)$ is still elliptic but possibly unbounded. Later in this section, we will give examples for the choice of $D$ and obtain bounds on $\bm{\varepsilon_0}$.

\begin{theorem}\label{Lemma_p}
Let $\cB$ be elliptic. If $Q$ is the orthogonal projection (with respect to $\|\cdot\|_\mu$) onto $D_0$, we have

$$\bm{p}^2 = \left(\frac{\bm{\varepsilon}}{\bm{\varepsilon_0}}\right)^2 \leq 1 + \frac{1}{\alpha_2^2}\|QBQ^{\perp}\|^2,$$

where $Q^{\perp} = 1-Q$, $B: V\rightarrow V$ is the linear operator $\phi \mapsto \cB(\cdot,\phi)$, and the operator norm is defined as $\|B\| = \sup_{\|x\|_\mu=1} \|Bx\|_\mu$.
\end{theorem}

{\bf Proof.} In Appendix \ref{app_lemma_p}.

Note that $\|Q B Q^\perp\| \leq \|QB\|$ is always finite even though $B$ is possibly unbounded since $Q$ is the projection onto a finite-dimensional subspace. The bottom line of Theorem \ref{Lemma_p} is that if $B$ leaves the subspace $D$ almost invariant, then $\hat\phi$ is almost the best-approximation of $\phi$ in $\|\cdot\|_\mu$. The following Lemma gives a more detailed description. In the following, we will write $\|\cdot\| = \|\cdot \|_\mu$ for convenience.

\begin{lemma}\label{Qperp}
Let

$$\delta_L := \max_k \|Q^\perp L \chi_k\|, \qquad \delta_f := \max_k \|Q^\perp \eps^{-1}f \chi_k\|$$

be the maximal projection error of the images of the $\chi_k$'s under $L$ and $f$ respectively. Then

$$\|QBQ^\perp\| = \|Q^\perp BQ\| \leq (\delta_L + \delta_f) \sqrt{\frac{n}{m}},$$

where $m$ is the smallest eigenvalue of $\hat M$.
\end{lemma}

\begin{proof}
The first statement is true since $A$ is essentially self-adjoint. For the second statement, first of all

$$\|Q^\perp BQ\| = \|Q^\perp (\eps^{-1}f-L)Q\| \leq \|Q^\perp \eps^{-1}f Q\| + \|Q^\perp LQ\|$$

holds from the triangle inequality. We now bound the term involving $L$. Notice that for $\hat\phi = \sum_i \hat \phi_i \chi_i \in D$:

$$\|Q^\perp L \hat\phi\| = \|\sum_i \hat \phi_i Q^\perp L\chi_i\| \leq \delta_L \sum_i |\hat \phi_i| = \delta_L \|\hat \phi\|_1.$$

Then, with $\hat M_{ij} :=\langle \chi_i,\chi_j\rangle_\mu$:

\begin{equation*}
\|Q^\perp L Q\| = \sup_{\phi = \phi_{||} + \phi_\perp \in V}\frac{\|Q^\perp L \phi_{||}\|}{\|\phi\|}\leq \sup_{\phi_{||} \in D}\frac{\|Q^\perp L \phi_{||}\|}{\|\phi_{||}\|} \leq \delta_L\sup_{\hat \phi\in\R^n}  \frac{\|\hat \phi\|_1}{\sqrt{\langle \hat \phi,\hat \phi\rangle_M}}
\end{equation*}

A similar result holds for  the term involving $f$. The statement now follows from a standard equivalence between finite-dimensional norms, $\|\hat \phi\|_1 \leq \sqrt{n}\|\hat \phi\|_2$, and the fact that $\hat M$ is symmetric, which implies that $\langle \hat \phi,\hat \phi\rangle_M = \hat \phi^T \hat M\hat \phi \geq m \hat \phi^T\hat \phi = m \|\hat \phi\|_2^2$. 
\end{proof}

To summarise, Theorem \ref{Lemma_p} and Lemma \ref{Qperp} give us a formula for the projection performance $p$ which states that 

$$\bm{p}^2 \leq 1 + \frac{n}{m}\frac{(\delta_L+\delta_f)^2}{\alpha^2_2}.$$

How large or small $\delta_f$ is will depend on the behaviour of $f$, if i.e. $f = \text{const}$ then $\delta_f=0$. Both $\delta_f$ and $\delta_L$ are always finite even though $L$ is possibly unbounded.

We comment on the best-approximation error $\bm{\varepsilon_0}$ for two choices for the subspace $D$ which will reappear again later in the paper.

\paragraph{Full partition}

Let $\cS$ be fully partitioned into disjoint sets $A_1,\ldots,A_{n+1}$ with centres $x_1,\ldots,x_{n+1}$ and such that $A_{n+1} := A$, and define $\chi_i := \chi_{A_i}$. These $\chi_i$ satisfy the assumptions (S1) and (S2) discussed in Section \ref{sec_Galerkin}. By definition we can bound $\bm{\varepsilon_0}$ by any interpolation $I\phi\in D$ of the solution $\phi$:

$$\bm{\varepsilon_0} \leq \|\phi - I\phi\|_\mu.$$

As interpolation, we choose $I\phi(x) = \sum_i c_i \chi_i(x)$ where $c_i = \frac{1}{\|\chi_i\|_1}\int_{A_i}\phi(x)d\mu$. If the $A_i$ are cubes of length $h$ and $\phi$ is twice continuous differentiable, then using standard techniques one can show that $\bm{\varepsilon_0}$ is linear in $h$, see e.g. \cite{braack2012}.

\paragraph{Incomplete partition}

Suppose the potential $V(x)$ has $n+1$ deep minima $x_1,\ldots,x_{n+1}$. Let $C_1,\ldots,C_{n+1}$ be convex 'core' sets around $x_1,\ldots,x_{n+1}$ and such that $A = C_{n+1}$. We write $C =  \cup_{i=1}^{n+1} C_i$ and $T=\cS\setminus C$ and introduce $\tau_C = \inf \{t\geq 0: X_t\in C\}$. We take $\chi_i$ to be the committor function associated to the set $C_i$, that is 

\begin{equation}
\chi_i(x) = \bP(X_{\tau_C} \in C_i| X_0 = x).
\label{eq_committor}
\end{equation}

These $\chi_i$ satisfy the assumptions (S1) and (S2). Since we do not have an order parameter $h$ controlling the resolution of the discretization, standard PDE techniques for bounding $\bm{\varepsilon_0}$ fail. Indeed, typically we will have very few basis functions compared to a grid-like discretization. The following Lemma gives a bound on $\bm{\varepsilon_0}$.

\begin{theorem}\label{lemma_best}
Let $Q$ be the orthogonal projection onto the subspace $D$ spanned by the committor functions (\ref{eq_committor}), and let $\phi$ be the solution of (\ref{eq_bvp1}). Then we have
 $$\bm{\varepsilon_0} = \|Q^\perp \phi\|_\mu \leq \|P^\perp \phi\|_\mu + \mu(T)^{1/2}\left[\kappa \|f\|_\infty + 2 \|P^\perp \phi\|_\infty \right]$$
 
 where $\|\cdot\| = \|\cdot\|_\mu$, $\kappa = \sup_{x\in T} \bE_x[\tau_{\cS\setminus T}]$, and $P$ is the orthogonal projection onto the subspace $V_c = \{v\in L^2(\cS,\mu), v=const \;\mbox{on every}\; C_i \} \subset L^2(\cS,\mu)$.
\end{theorem}
\begin{proof}
 In Appendix \ref{app_best}
\end{proof}

In theorem \ref{lemma_best}, $\kappa$ is the maximum expected time of hitting the metastable set from outside (which is short). Note further that $P^\perp \phi = 0$ on $T$. The errors $\|P^\perp \phi\|_\mu$ and $\|P^\perp \phi\|_\infty$ measure how constant the solution $\phi$ is on the core sets. Theorem \ref{lemma_best} gives us excellent control over $\bm{\varepsilon_0}$, and together with theorem \ref{Lemma_p} we have full control over the discretization error $\eps$ for the case of incomplete partitions. These error bounds are along the lines of MSM projection error bounds \cite{Sarich2010}, \cite{Djurdjevac2012}, and to the best of the authors' knowledge they are new.

\begin{remark}
It would be nice to have an error estimate also for the value function. In general such an estimate is difficult to get, because of the nonlinear logarithmic transformation   $W=-\eps\log\phi$ involved. However we know that $\phi$ and its discrete approximation are both uniformly bounded and bounded away from zero. Hence the logarithmic transformation is uniformly Lipschitz continuous on its domain, which implies that the $L^{2}$ error bounds holds for the value function with an additional prefactor given by the Lipschitz constant squared; for a related argument see \cite{Zhang2013}  
\end{remark}

\subsection{Interpretation in terms of a Markov decision problem}\label{sec_interpretation}

We derive an interpretation of the discretized equation (\ref{eq_FK_disc}) in terms of a MJP. We introduce the diagonal matrix $\Lambda$ with entries $\Lambda_{ii} = \sum_j F_{ij}$ (zero otherwise) and the full matrix $G=K -\eps^{-1}(F-\Lambda)$, and rearrange (\ref{eq_FK_disc}) as follows:
\begin{equation} \label{eq_stoch2}
\begin{aligned}
\sum_{j=1}^{n+1} \left(G_{ij} - \eps^{-1} \Lambda_{ij}\right) \hat{\phi}_j  & = 0\,, \quad i\in \{1,\ldots,n\}\\
\hat{\phi}_{n+1} & = 1\,,
\end{aligned}
\end{equation}

This equation can be given a stochastic interpretation. To this end let us introduce the vector $\pi\in\R^{n+1}$ with nonnegative entries $\pi_i = \langle \chi_i,\mathds{1}\rangle$ and notice that $\sum_i \pi_i = 1$ follows immediately from the fact that the basis functions $\chi_i$ form a partition of unity, i.e. $\sum_i \chi_i = \mathds{1}$. This implies that $\pi$ is a probability distribution on the discrete state space $\hat{\cS} = \{1,\ldots,n+1\}$.
We summarise properties of the matrices $K$, $F$ and $G$:
\begin{lemma} Let $K$, $G$, $F$ and $\pi$ be as above.
\label{lemma_mat}
\begin{enumerate}
\item[(i)] $K$ is a generator matrix (i.e. $K$ is a real-valued square matrix with row sum zero and positive off-diagonal entries) with stationary distribution $\pi$ that satisfies detailed balance 
\[
\pi_{i}K_{ij} = \pi_{j}K_{ji}\,,\quad i,j\in \hat{\cS}
\]
\item[(ii)] $F \geq 0$ (entry-wise) with $\pi_{i}F_{ij} = \pi_{j} F_{ji}$ for all  $i,j\in \hat{\cS}$.
\item[(iii)] $G$ has row sum zero and satisfies $\pi^T G = 0$ and $\pi_{i}G_{ij} = \pi_{j}G_{ji}$ for all $i,j\in \hat{\cS}$.
\item[(iv)] There exists a (possibly $\eps$-dependent) constant $0<C<\infty$ such that $G_{ij} \geq 0$ for all $i\neq j$ if $\|f\|_{\infty}\le C$. In this case equation (\ref{eq_stoch2}) admits a unique and strictly positive solution $\hat{\phi} > 0$.
\end{enumerate}
\end{lemma}

\begin{proof}
$(i)$ follows from $\sum_i \chi_i(x) = \mathds{1}$ and reversibility of $L$: We have $\sum_i \pi(i) K_{ij} = \sum_i \langle \chi_i, L\chi_j\rangle_\mu = \langle L1, \chi_j\rangle_\mu = 0$ and $\pi(i)K_{ij} = \langle \chi_i, L\chi_j\rangle_\mu = \langle L\chi_i,\chi_j\rangle_\mu = \pi(j) K_{ji}$. $(ii)$ follows from $f(x)$ being real and positive for all $x$. As for $(iii)$, $G$ has row sum zero by (i) and the definition of $\Lambda$. $\pi(i)G_{ij} = \pi(j)G_{ji}$ follows from $(i)$, $(ii)$ and the fact that $\Lambda$ is diagonal, and $\pi^T G = 0$ follows directly. For $(iv)$, rewrite  (\ref{eq_stoch2}) as the $n\times n$-system $\bar G_\lambda\bar\phi = g$ where $\bar G_\lambda$ is the first $n$ rows and columns of $G_\lambda:= -G + \eps^{-1} \Lambda$, $\hat\phi = (\bar\phi,1)^T$ and $-g$ is the vector of the first $n$ entries of the $(n+1)$st row of $G_\lambda$. Choose $C$ such that $\eps^{-1} \langle \chi_i, f \chi_j\rangle_\mu \leq \langle \chi_i, L\chi_j\rangle_\mu$ for all $i\neq j$. Then $g > 0$ and $\bar G_\lambda$ is a 
non-singular $M$-matrix and thus inverse monotone \cite{Berman1979}, that is from $\bar G_\lambda \bar\phi = g$ and $g > 0$ follows $\bar \phi > 0$.
\end{proof}

It follows that if the running costs $f$ are such that $(iv)$ in Lemma \ref{lemma_mat} holds, then $G$ is a generator matrix of a MJP that we shall denote by $(\hat{X}_t)_{t\geq 0}$, and by lemma \ref{lemma_backw_mjp}, (\ref{eq_stoch2}) has a unique and positive solution of the form

$$\hat\phi(i) = \bE\left[\exp\left(-\eps^{-1}\int_0^{\tau_A} \hat f(\hat X_s)ds\right) \middle| \hat X_0 = i\right]$$

with $\hat f(i) = \Lambda_{ii}$ and $\tau_A = \inf\{t\geq 0| \hat X_t = i+1\}$. In fact (\ref{eq_stoch2}) can be interpreted as the backward Kolmogorov equation for $\hat\phi$. Moreover, the logarithmic transformation $\hat{W} = -\eps\log\hat{\phi}$ is well-defined and can be interpreted as the value function of the Markov decision problem (\ref{MJP_HJB})--(\ref{MJP_cost}), that is, we seek to minimize
\[
 \hat J(v;i) = \bE\left[\int_0^{\tau_A}\left(\hat{f}(\hat{X}^{v}_s) + k^{v}(\hat{X}^{v}_s)\right)ds \middle| \hat X^{v}_0 = i\right]
\]
over Markov control strategies $v: \hat\cS\rightarrow (0,\infty)$ with the costs
\begin{equation*} \label{eq_runningcost}
\hat{f}(i) = \Lambda_{ii}\,,\quad k^{v}(i) = \eps\sum_{j\neq i} G_{ij}\left\{ \frac{v(j)}{v(i)}\left[\log \frac{v(j)}{v(i)} - 1\right] + 1\right\}.
\end{equation*}

This completes the construction of the discrete control problem. We now analyse it in detail for the two choices of projection subspace $D$ introduced before.

\paragraph{Full partitions}

We partition $\cS$ into disjoint sets $A_i$ that we take to be rectangular with centres $x_i$, we let $S_{ij} = A_i\cup A_j$ and $h_{ij}$ be the line joining $x_i$ and $x_j$, see Figure \ref{fig_fva_mesh}. Let $m(A_i)$, $m(S_{ij})$ and $m(h_{ij})$ be the Lebesgue volumes of the cells $A_i$, surfaces $S_{ij}$ and lines $h_{ij}$ respectively, and let $\bar x_{ij} = S_{ij}\cap h_{ij}$.

\begin{figure}[ht]
 \centering
 \def\svgwidth{0.4\columnwidth}
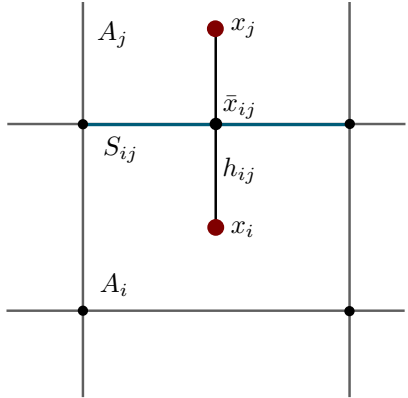
\caption{The mesh for the full partition.}
\label{fig_fva_mesh}
\end{figure}

We show in Appendix \ref{app_FVA} that the matrix $K_{ij}$ has then components

\begin{equation}
 K_{ij} \approx \frac{1}{\Delta_{ij}} e^{-\eps^{-1} (V(\bar x_{ij})-V(x_i))}, \quad \Delta_{ij}^{-1} = \eps \frac{m(S_{ij})}{m(h_{ij})m(A_i)}
\label{FVA_gen}
\end{equation}

if $i$ and $j$ are neighbours ($K_{ij} = 0$ otherwise). $K$ is the generator of a MJP on the cells $A_i$ and coincides with the so-called \emph{finite volume approximation} of $L$ discussed in \cite{Latorre2011}. The approximations we use in Appendix \ref{app_FVA} to calculate the integrals coincide with the ones used in \cite{Latorre2011}. The invariant distribution of $K$ is

\begin{equation}
 \pi_i := \|\chi_i\|_1 = \int_{A_i} d\mu \approx m(A_i)e^{-\beta V(x_i)}.\label{eq_FVA_dist}
\end{equation}

Since $\chi_i \cap \chi_j = \emptyset$ for $i\neq j$, $F$ is diagonal, and we obtain the running costs

\begin{equation}
 \hat f(i) = \frac{1}{\pi_i}\int_{A_i}f(x)\mu(x)dx = \bE_\mu[f(X_t)| X_t\in A_i]
\label{eq_lambda}
\end{equation}

by simply averaging $f(x)$ over the cell $A_i$. (\ref{eq_lambda}) is also a sampling formula for $\hat f(i)$. It follows directly that $G=K$, and in particular assumption (iv) of Lemma \ref{lemma_mat} holds for any $f$.

$K$ and $\pi$ can be computed from the potential $V$ and the geometry of the mesh. In fact we can also derive a relation to standard Markov state modelling. Let $P^\tau$ be the MSM transition matrix with lag time $\tau$ associated to the partition $\{A_1,\ldots,A_{n+1}\}$, that is:

\begin{equation*}
 P^\tau_{ij} = \frac{1}{\|\chi_i\|_1}\langle \chi_i, T_\tau \chi_j\rangle_\mu, \quad T_\tau = \exp(\tau L)\,.
\end{equation*}

Then, by bounded convergence,

\begin{equation}
\lim_{\tau\rightarrow 0} \frac{1}{\tau}\left(P^\tau_{ij} - \delta_{ij}\right) = \lim_{\tau\rightarrow 0}\frac{1}{\pi_i}\langle \chi_i, \frac{1}{\tau}(T_\tau - \mathds{1})\chi_j\rangle = \frac{1}{\pi_i}\langle \chi_i, L\chi_j\rangle = K_{ij},
\label{eq_fullMSM}
\end{equation}

thus $K$ is the generator of the semigroup of transition matrices $P^\tau$. For finite lag time $\tau$ the transition matrices $P^\tau$ can be sampled from long realizations of the original dynamics. This introduces a sampling error which depends on details of the partition, the available sampling data and the existence of rare transitions in the system. We do not address the sampling error in this paper, see e.g. \cite{Prinz,Roeblitz08}. In view of (\ref{eq_fullMSM}) we could in principle sample $K$ by sampling $P^\tau$ for very small $\tau$, this is difficult however due to recrossing problems, see e.g. \cite{Chodera2011}.

\paragraph{Recovering Markov Chain approximations}

In Appendix \ref{app_MCA} we show that, if $\cS$ is one-dimensional\footnote{These assumptions are mostly for notational convenience. The proof should be straightforward to generalise.} and the cells are intervals of length $h$, we can write the nonzero off-diagonal components of the generator of the controlled MJP as

\begin{eqnarray}
 G^v_{i,i \pm 1} & = & \frac{1}{h^2}\left(\eps \mp \frac{h}{2}\left(\nabla V(x_i) - \alpha_v(i)\right) + \mathcal{O}(h^2)\right), \notag \\ 
 \alpha_v(i) & = & \frac{\eps}{h}\left(\log v(i+1) - \log v(i-1)\right) \label{eq_mca1}
\end{eqnarray}

where as usual $G^v_{ii} = -\sum_{j\neq i} G^v_{ij}$. We also show that the running costs of strategy $v$ can be written as

\begin{equation}
 k^v(i) = \frac{1}{4}\alpha_v^2(i) + \mathcal{O}(h). \label{eq_mca_cost}
\end{equation}

This may be compared to a well-known discretization of continuous optimal control problems known as the Markov chain approximation (MCA); see \cite{Kushner1992}. The MCA discretization may be obtained by replacing derivatives with finite differences in the continuous control problem (\ref{hjb}). The result is

\begin{equation}
 \min_{\tilde\alpha(i)\in \R} \left[(\tilde G^{\tilde\alpha} \tilde W)(i) + \frac{1}{4}\tilde\alpha^2(i)  + \tilde f(i)\right] = 0 \;\mbox{for}\; i\in \{1,\ldots,n\}, \quad \hat W(n+1) = 0 \label{eq_mca2}
\end{equation}

which is a Bellman equation for the MCA optimal cost $\tilde W$ with strategies $\tilde \alpha \in \R^n$ and average running costs $\hat f(i)$. The nonzero components of the MCA generator\footnote{In the literature, one usually considers the matrix $I+\tilde G^{\tilde\alpha}$ and interprets it as a transition matrix for a Markov chain. To be able to compare with our approach, we instead interpret $\tilde G^{\tilde\alpha}$ as a generator, which is equivalent.} $\tilde G^{\tilde\alpha}$ corresponding to the strategy vector $\tilde\alpha$ are

\begin{equation}
  \tilde G^{\tilde\alpha}_{i,i\pm 1} = \frac{1}{h^2}\left(\eps - \frac{h}{2}\left(\nabla V(i) - \tilde\alpha_i\right)\right), \quad \tilde G_{ii} = -\sum_{j\neq i}\tilde G_{ij}.
\label{mca_matrix}
\end{equation}

To compare both control problems, we need to be able to compare strategies. For our MJP control problem, strategies $v$ were positive functions on $\hat{\cS}$, but with $w = \log v$ we can think of $U$ as $\R^{n+1}$. For the MCA approximation, $\tilde U = \R^n$. (\ref{eq_mca1}) gives a mapping $z: U\rightarrow \tilde U$ with $z(v) = \alpha_v$. It can be shown that $z$ is onto and can therefore be used to map strategies. Now, comparing (\ref{eq_mca2}) and (\ref{mca_matrix}) with (\ref{MJP_HJB}) and (\ref{eq_mca1}) gives $G^v = \tilde G^{\tilde\alpha}\left(1 + \mathcal{O}(h^2)\right)$ if we set $\tilde{\alpha} = z(v) = \alpha_v$, and the Bellman equations are equal up to first order in $h$ if strategies are mapped accordingly. Moreover, optimal strategies have the same functional dependence on optimal costs:

\begin{eqnarray*}
 \tilde\alpha^* & = & -\frac{1}{h}\left(\tilde W(i+1) - \tilde W(i-1)\right) \\ 
 \alpha_{v^*} & = & \frac{\eps}{h}\left(\log v^*(i+1) - \log v^*(i-1)\right) = -\frac{1}{h}\left(\hat W(i+1) - \hat W(i-1)\right).
\end{eqnarray*}

In the limit $h\rightarrow 0$, our discretization therefore coincides with the MCA. Convergence theory for MCAs \cite{Kushner1992} states that the discrete control problem (\ref{eq_mca2}) converges for $h\rightarrow 0$ to the continuous problem (\ref{hjb}) and that $\tilde W \rightarrow W$ and $\tilde \alpha \rightarrow \alpha$ converge weakly in $V$. Therefore we can deduce that $\hat W\rightarrow W$ and $\alpha_v^*\rightarrow \alpha$ converge weakly in $V$ as $h\rightarrow 0$.

\paragraph{Incomplete partitions}

We use a core set partition of $\cS$ as introduced in Section \ref{sec_Galerkin}. The projection onto the committor basis $\chi_i$ also allows for a stochastic interpretation in terms of the forward and backward milestoning process $\tilde X_t^\pm$, which we define in the following way: $\tilde X_t^+ = i$ if the process $X_t$ visits the core set $C_i$ next, and $\tilde X_t^- = i$ if $X_t$ came from $C_i$ last. With this definition, the discrete costs can be written as
\begin{equation}
   \hat f(i) = \frac{1}{\pi_i}\langle \chi_i,f\sum_j \chi_j\rangle = \int \nu_i(x)f(x)dx = \bE_\mu\left[f(X_t)\middle|\tilde X_t^{-} = i\right]
\end{equation}

where $\nu_i(x) = \pi_{i}^{-1}\chi_i(x)\mu(x) = \bP(X_t = x| \tilde X_t^{-} = i)$ is the probability density of finding the system in state $x$ given that it came last from $i$. Hence $\hat f(i)$ is the average costs conditioned on the information $\tilde X^{-}_t = i$, i.e. $X_t$ came last from $A_i$, which is the natural extension to the full partition case where $\hat f(i)$ was the average costs conditioned on the information that $X_t\in A_i$.

The matrix $K_{ij} = \pi_i^{-1}\langle \chi_i,L \chi_j\rangle$ is reversible with stationary distribution 

$$\pi_i = \langle \chi_i, \mathds{1} \rangle = \bP_\mu(\tilde X_t^- = i)$$ 

and is related to so called {\em core MSMs}. To see this, define the core MSM transition matrix $P^\tau$ with components $P^\tau_{ij} = \bP(\tilde X_{t+\tau}^+ = j|\tilde X_t^- = i)$, and the mass matrix $M$ with components $M_{ij} = \bP(\tilde X_{t}^+ = j|\tilde X_t^- = i)$. Then, it is not hard to show that for reversible processes we have $P^\tau_{ij} = \pi_i^{-1}\langle \chi_i, T^\tau \chi_j\rangle_\mu$ and $M_{ij} = \pi_i^{-1}\langle \chi_i,\chi_j\rangle_\mu$ so that

$$K = \frac{1}{\pi_i}\langle \chi_i,L\chi_j\rangle_\mu =  \lim_{\tau\rightarrow 0 } \frac{1}{\tau}\left(P^\tau - M\right).$$

Formally, $K$ is the generator of the $P^\tau$, but these do not form a semigroup since $M\neq \mathds{1}$, and therefore we cannot interpret $K$ directly as e.g. the generator of $\tilde X_t^-$. Nevertheless, the entries of $K$ are the transition rates between the core sets as defined in transition path theory \cite{Vanden-Eijnden2006}. We can obtain $P^\tau$ and $M$ from sampling as in the full partition case. The difference is that if the core sets are chosen as the metastable states of the system, $P^\tau$ can be sampled for all lag times $\tau$, and $K$ can be sampled directly. See \cite{Sarich2010}, \cite{Djurdjevac2010} and \cite{Schuette2011} for more details on the construction and sampling of core MSMs. In Appendix \ref{app_samplingF} we show that $F$ can also be sampled using

\begin{equation}
F_{ij} = \bE_\mu\left[f(X_t) \chi_{\{\tilde X_t^{+}= j\}}\middle| \tilde X_t^- = i\right]\label{eq_samplingF}
\end{equation}

Therefore, as in the construction of core MSMs, we do not need to compute committor functions explicitly. Note however that $G \neq K$, there is a reweighting due to the overlap of the $\chi_i$'s which causes $F$ to be non-diagonal. This reweighting is the surprising bit of this discretization. From Lemma \ref{lemma_mat} we see however that $G$ and $K$ are both reversible with stationary distribution $\pi$. Finally, note that if the cost function $f(x)$ doesn't satisfy $\|f\|_\infty \leq C$ from (iv) in Lemma \ref{lemma_mat}, $G$ will not even be a generator matrix. In this case (\ref{eq_FK_disc}) still has a solution $\hat \phi$ which is the best-approximation to $\phi$, but this solution may not be unique, it may not satisfy $\hat \phi > 0$, and we have no interpretation as a discrete control problem.

\section{Numerical Results}

We will present two examples to illustrate the approximation of LQ-type stochastic control problems based on a sparse Galerkin approximation using MSMs.   

\subsection{1D triple well potential}

To begin with we study diffusion in the triple well potential which is presented in Figure \ref{fig_num1pota}. This potential has three minima at approximately $x_{0/1} = \pm 3.4$ and $x_2 = 0$. We choose the three core sets $C_i = [x_i-\delta,x_i+\delta]$ around the minima with $\delta = 0.2$. We choose $C_0 = A$ as the target set and the running cost $f=\sigma = const$, such that the control goal is to steer the particle into $C_0$ in minimum time.

In Figure \ref{fig_num1pota} the potential $V$ and effective potential $U$ are shown for $\eps = 0.5$ and $\sigma = 0.08$ (solid lines), cf. equation (\ref{SDE_controlled}). One can observe that the optimal control effectively lifts the second and third well up which means that the optimal control will drive the system into $C_0$ very quickly. The reference computations here have been carried out using a full partition FEM discretization of (\ref{linbvp}) with a lattice spacing of $h=0.01$. Now we study the MJP approximation constructed via the committor functions shown in Figure \ref{fig_num1potb}. These span a three-dimensional subspace, but due to the boundary conditions the subspace $D_0$ of the method is actually two-dimensional. The dashed line in Figure \ref{fig_num1pota} gives the approximation to $U$ calculated by solving (\ref{eq_stoch2}). We can observe extremely good approximation quality, even in the transition region. In Figure \ref{fig_num1potb_1} the optimal control $u^*(x)$ (solid line) and 
its approximation $\hat u^* = -2\nabla\hat W$ (dashed line) are shown. The core sets are shown in blue. We can observe jumps in $\hat u^*$ at the left boundaries of the core sets. This is to be expected and comes from the fact that the committor functions are not smooth at the boundaries of the core sets, but only continuous.

Next we construct a core MSM to sample the matrices $K$ and $F$. 100 trajectories of length $T=20000$ were used to build the MSM. In Figure \ref{fig_num1potd}, $W$ and its estimate using the core MSM is shown for $\epsilon = 0.5$ and different values of $\sigma$. Each of the 100 trajectories has seen about four transitions. For comparison, a direct sampling estimate of $W$ using the same data is shown (green). The direct sampling estimate suffers from a large bias and variance and is practically useless. In contrast, the MSM estimator for $W$ performs well for all considered values of $\sigma$. The constant $C$ which ensures $\hat\phi>0$ when $\sigma\leq C$ is approximately $0.2$ in this case. This seems restrictive but still allows to capture all interesting information about $\phi$ and $W$.

\begin{figure}[ht]
 \begin{subfigure}[b]{0.48\textwidth}
 \includegraphics[width=0.98\textwidth]{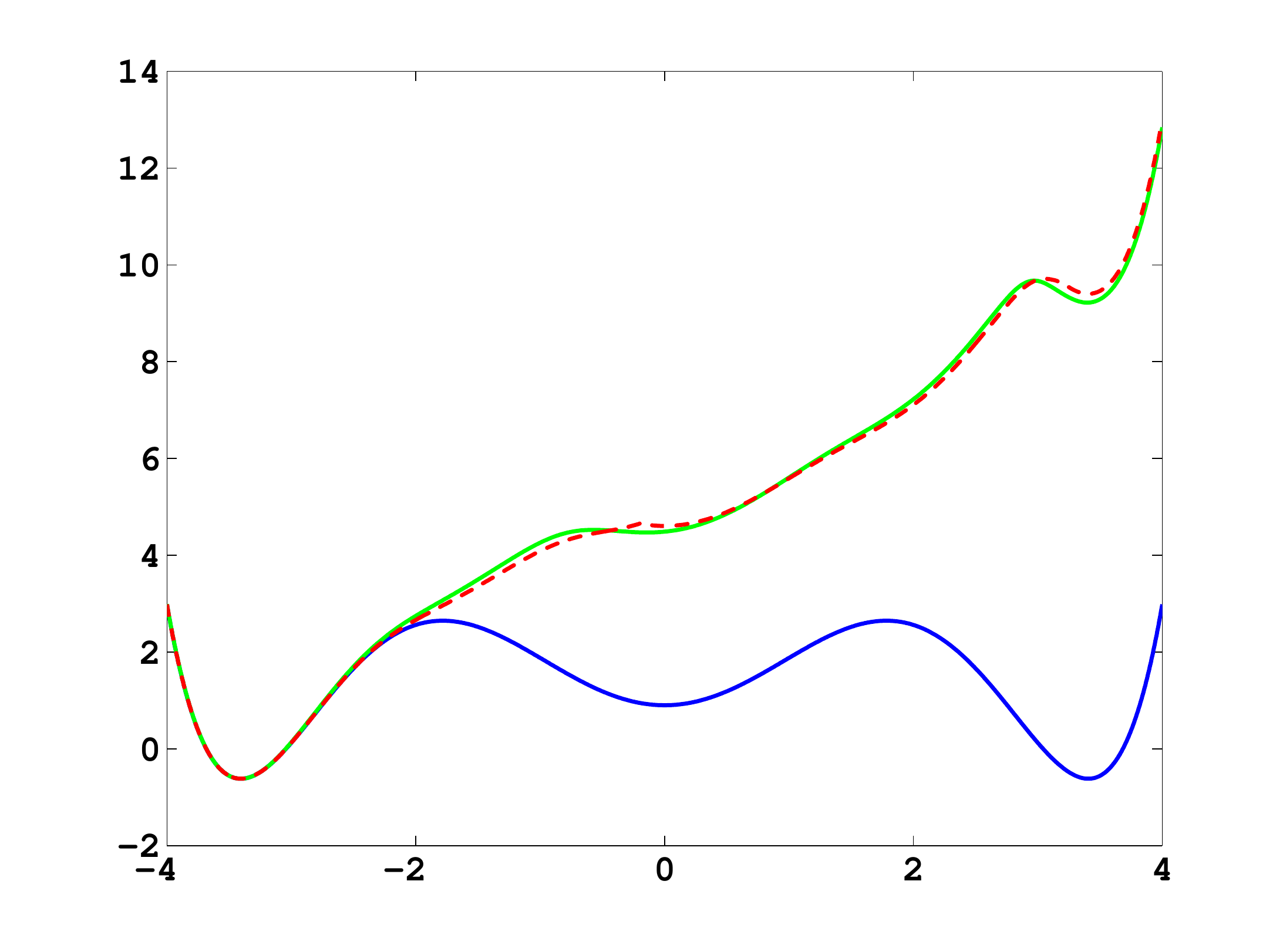}
 \caption{}
 \label{fig_num1pota}
 \end{subfigure}
 \begin{subfigure}[b]{0.48\textwidth}
 \includegraphics[width=0.98\textwidth]{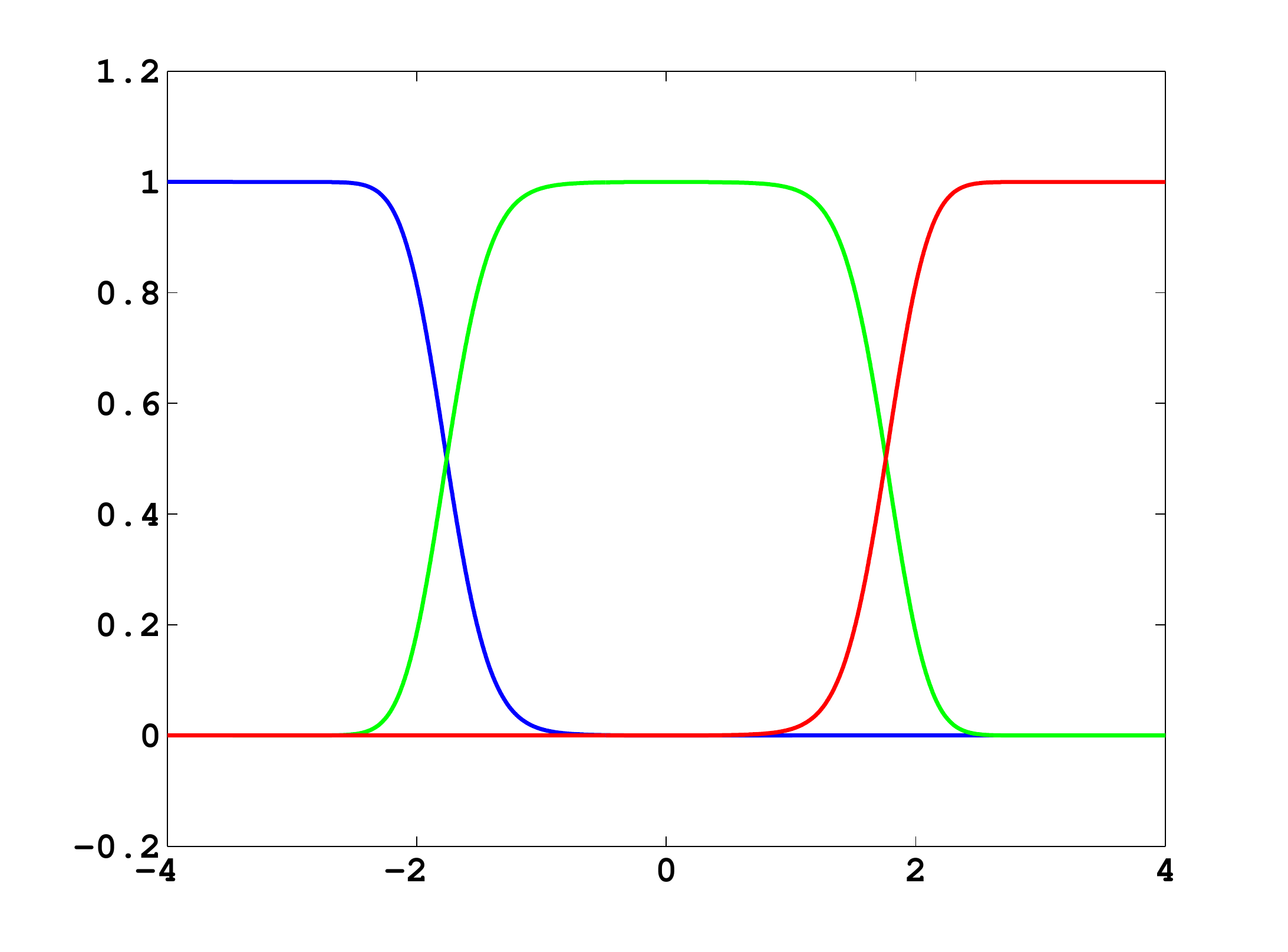}
 \caption{}
 \label{fig_num1potb}
 \end{subfigure}
 \\
 \begin{subfigure}[b]{0.48\textwidth}
 \includegraphics[width=0.98\textwidth]{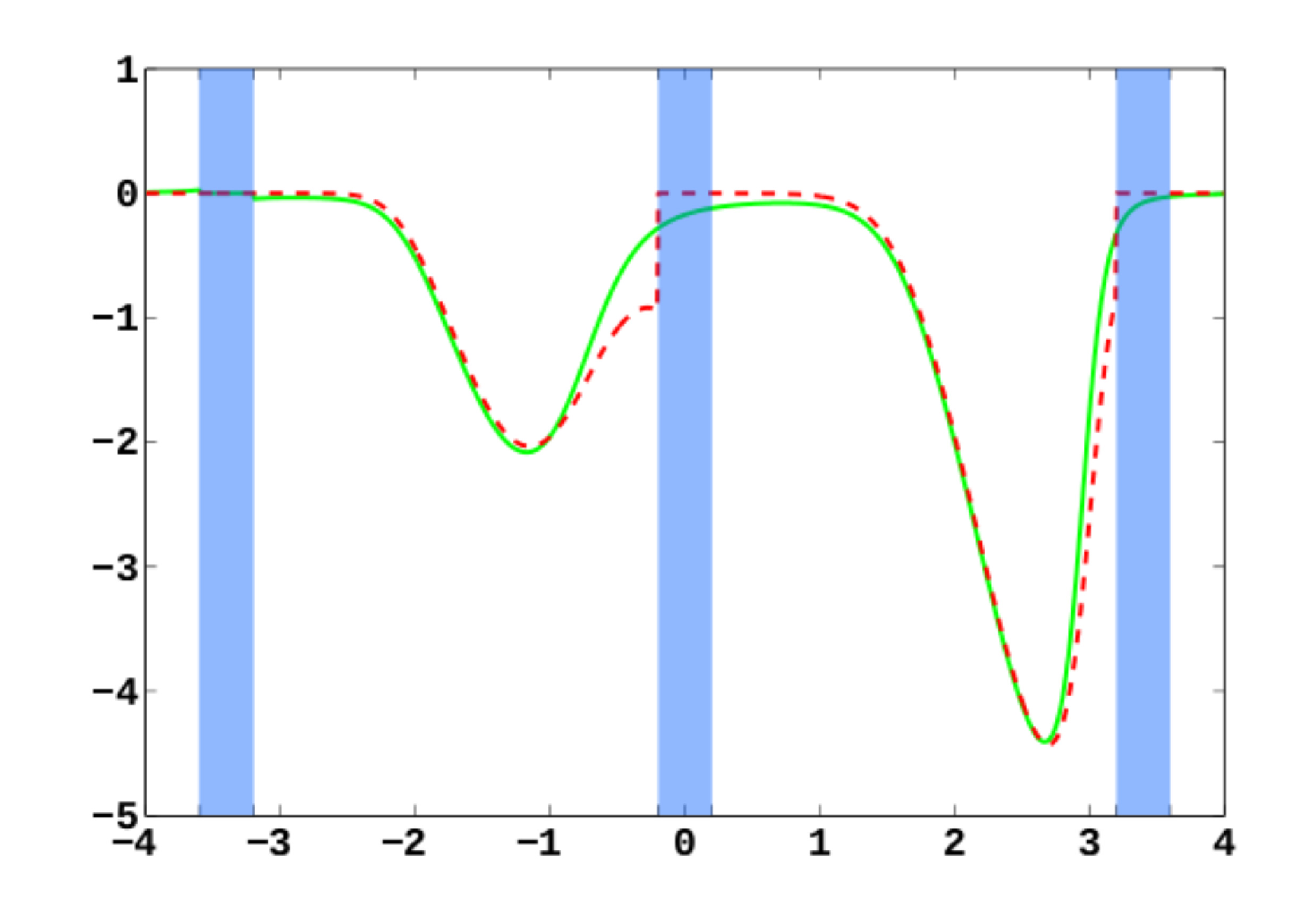}
 \caption{}
 \label{fig_num1potb_1}
 \end{subfigure}
 \begin{subfigure}[b]{0.48\textwidth}
 \includegraphics[width=0.98\textwidth]{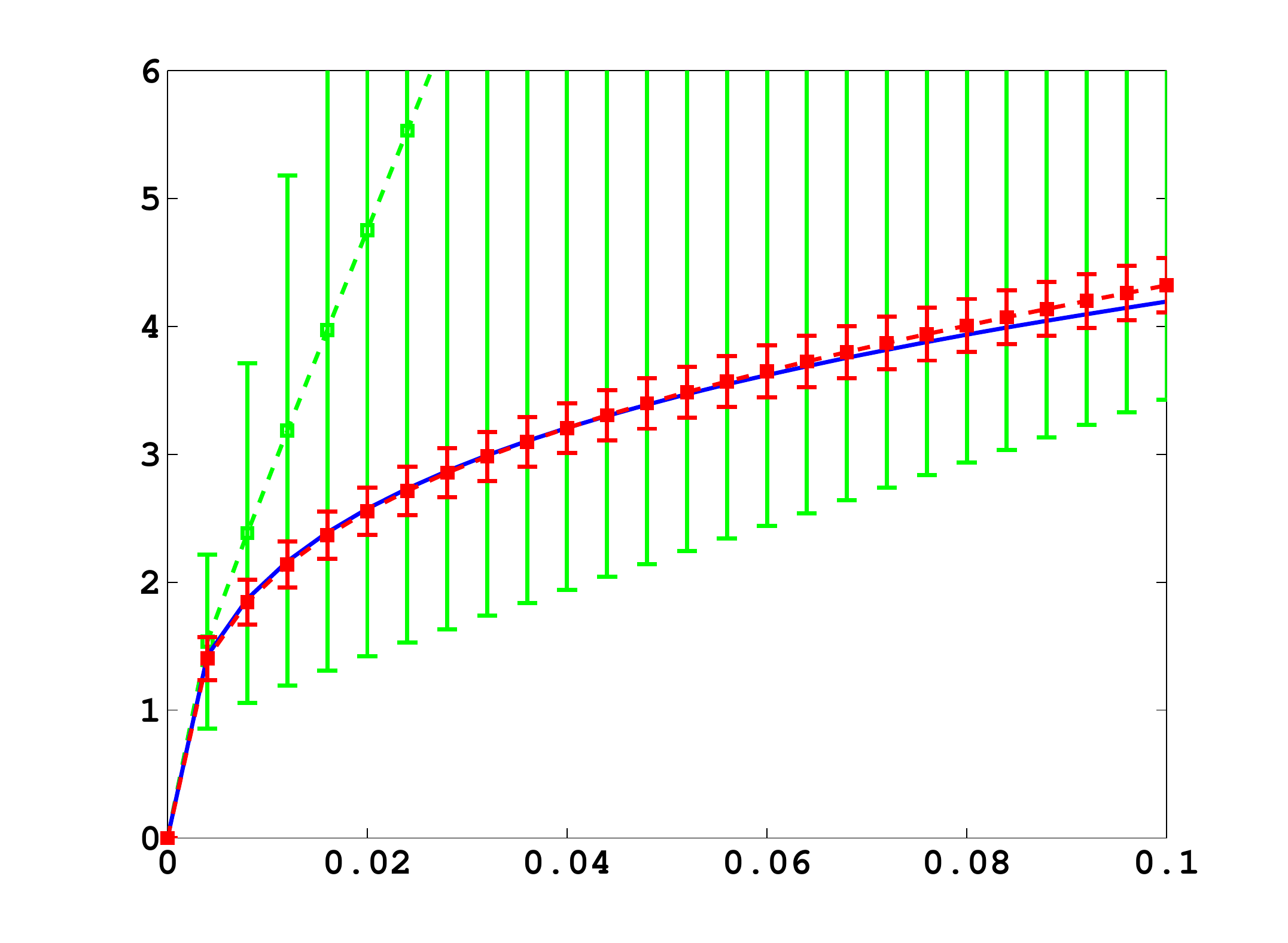}
 \caption{}
 \label{fig_num1potd}
 \end{subfigure}
 \caption{Three well potential example for $\epsilon = 0.5$ and $\sigma = 0.08$. (\subref{fig_num1pota}) Potential $V(x)$ (blue), effective potential $U = V + 2W$ (green) and approximation of $U$ with committors (dashed red). (\subref{fig_num1potb}) The three committors $\chi_1(x)$, $\chi_2(x)$ and $\chi_3(x)$. (\subref{fig_num1potb_1}) The optimal control $\alpha^*(x)$ (solid line) and its approximation (dashed line). Core sets are shown in blue. (\subref{fig_num1potd}) Optimal cost $W$ for $\eps = 0.5$ as a function of $\sigma$. Blue: Exact solution. Red: Core MSM estimate. Green: Direct sampling estimate.}
 \label{fig_num1pot}
\end{figure}

\FloatBarrier

\subsection{Alanine dipeptide}

As a second, non-trivial example we study the $\alpha$-$\beta$ conformational transition in Alanine dipeptide (ADP), a well-studied test system for molecular dynamics applications. We use a $1\mu s$ long molecular dynamics trajectory simulated in a box of 256 (explicit) water molecules using the CHARMM27 force field. The conformational dynamics is monitored as usual via the backbone dihedral angles $\phi$ and $\psi$. The data was first presented in \cite{Schuette2011}. In Figure \ref{fig_alanine}, a cartoon of the molecule is shown.

\begin{figure}[ht]
 \centering
 \includegraphics[width=0.3\textwidth]{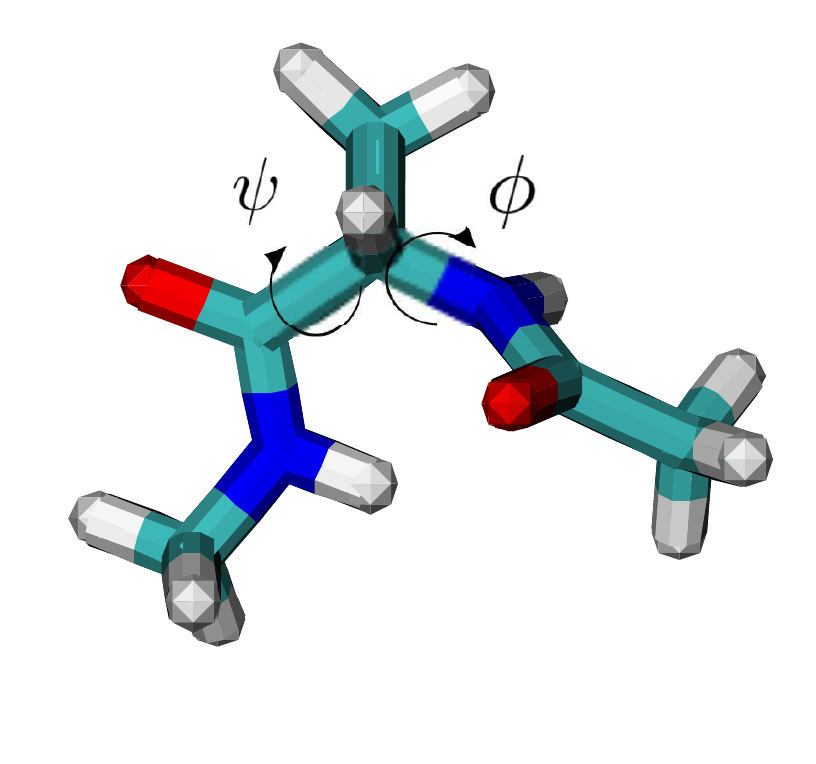}
 \caption{Alanine dipeptide.}
 \label{fig_alanine}
\end{figure}

\paragraph{Validation of the MSM approximation}

We construct a full partition MSM with 250 clusters using $k$-means clustering. The cluster algorithm uses the Cartesian coordinates of the ADP configurations as input data, ignoring the ADP velocities and the solvent molecules. Our first test is study the effect of the approximation of the generator matrix $K$ by the sampled transition matrix $P^{\tau}$ according to $\tau^{-1}(P^\tau - 1)$; see (\ref{eq_fullMSM}). To obtain a robust estimate of $K$, we first focus on the mean first passage time (MFPT) $t(x) = \bE_x[\tau_\alpha]$ where $\tau_\alpha$ is the first hitting time of the $\alpha$ conformation, which we define as a ball $C_\alpha$ with radius $r=45$ around the known minimum $(\phi_\alpha,\psi_\alpha) = (-80,-60)$ of the free energy landscape in $(\phi,\psi)$. The MFPT satisfies the matrix equation
$$K\hat t = -1 \;\mbox{outside }\, C_\alpha,\quad \hat t=0\;\mbox{in}\;C_\alpha$$

which we study with $K$ replaced by $\tau^{-1}(P^ \tau - 1)$. In Figure \ref{fig_num3pota}, the results are shown for $\tau = 5ps$, we can identify the $\beta$-structure as the red cloud of clusters where $t(x)$ is approximately constant. In \ref{fig_num3potb}, $\hat t_{\beta\alpha} = \bE(\hat t(i)|i\in \beta)$ is shown as a function of $\tau$. We observe a linear behaviour for large $\tau$ which is due to the linear error introduced in the replacement of $K$ with $\tau^{-1}\left(P^\tau - 1\right)$ and a nonlinear drop for small $\tau$ which is due to non-Markovianity. Our best guess is therefore a linear interpolation to $\tau = 0$, which is indicated by the solid line. The result is $\hat t^{0}_{\beta\alpha} = 35.5ps$. As a comparison the reference value $\hat{t}^{\rm ref}_{\beta\alpha} = 36.1 ps$ from \cite{Schuette2011} is shown as a dashed line, that was computed therein as an inverse rate, using the slowest \emph{implied time scale (ITS)} and information about the equilibrium weights of the $\alpha$ and $\beta$ 
structure. We see very good agreement, which indicates that the strategy of linearly interpolating lag time dependent results to $\tau=0$ is robust.

\paragraph{Controlled transition to the $\alpha$-helical structure}
Next we consider an optimal control problem for steering the molecule into the $\alpha$-structure. We choose as the target region $A=C_\alpha$ and define running costs in the $(\phi,\psi)$ variables as $f(\phi,\psi) = f_0 + f_1 \|\psi-\psi_\alpha\|^2$ where $f_0$ and $f_1$ are constants and $\|\cdot\|$ is a simple metric on the torus. We choose $f_0 = 0.01$ and $f_1 = 0.001$, which represents a mild penalty for being away from the target region. We discretize this control problem using the same partition as for the MSM construction above. The matrix $K$ is again replaced by $\tau^{-1}(P^\tau - 1)$, the matrix $F$ is diagonal and can be sampled straightforwardly. The resulting generator matrix $G^{v^*}$ of the optimally controlled process can be used to compute the MFPT $\hat t^{(v^*)}$ of the controlled process according to the matrix equation
$$G^{v^*}\hat t^{(v^*)} = -1 \;\mbox{outside }\, C_\alpha,\quad \hat t^{(v^*)}=0\;\mbox{in}\;C_\alpha.$$
The results will again depend on the lag time $\tau$. Figure \ref{fig_num3potc} shows the results for $\tau = 5ps$, while \ref{fig_num3potd} shows the MFPT for different lag times and a linear interpolation to $\tau = 0$. We observe that the control leads to a speedup of the MFPT by 1--2 orders of magnitude. A larger speedup could easily be achieved by increasing the relative weight of $f$, compared to the quadratic penalization of the force. 

Figures \ref{fig_num3pote} and \ref{fig_num3potf} show the optimal cost $\hat W$ and optimal strategy $v^*$ for this problem. The optimal control $v^*$ is best understood in terms of the jump rates 
\[
G^{v^*}_{ij} = \frac{G_{ij}v^*(j)}{v^*(i)}\,.
\] 
If $v^*(i)$ is low, the controller accelerates jumps out of state $i$ while slowing down jumps into state $i$, and vice versa if $v^*(i)$ is high. The red cloud in Figure~\ref{fig_num3potf} actually has value 1, in accordance with the boundary conditions for $v^*$.

\begin{figure}[ht]
 \begin{subfigure}[b]{0.48\textwidth}
 \includegraphics[width=0.98\textwidth]{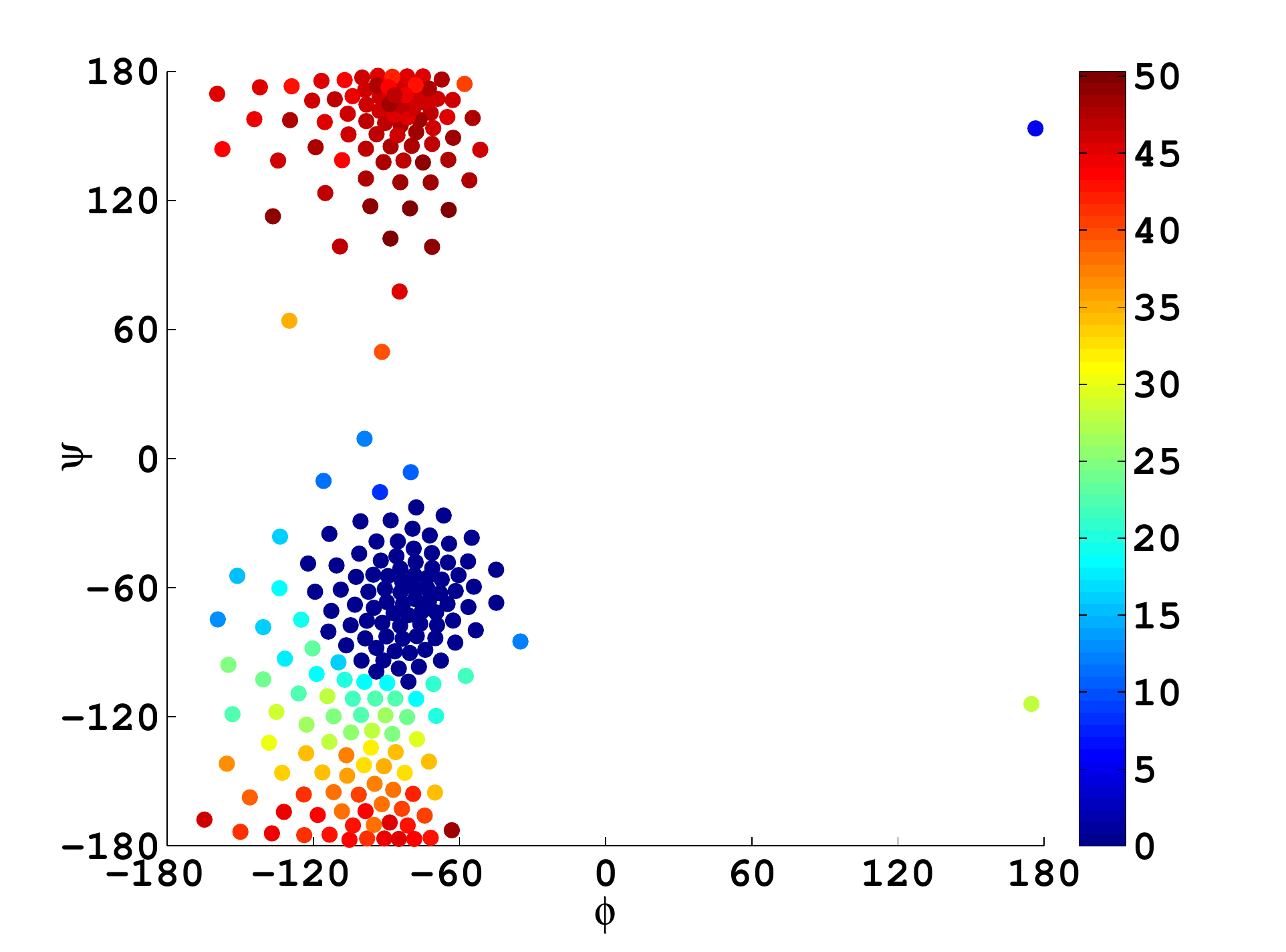}
 \caption{}
 \label{fig_num3pota}
 \end{subfigure}
 \begin{subfigure}[b]{0.48\textwidth}
 \includegraphics[width=0.98\textwidth]{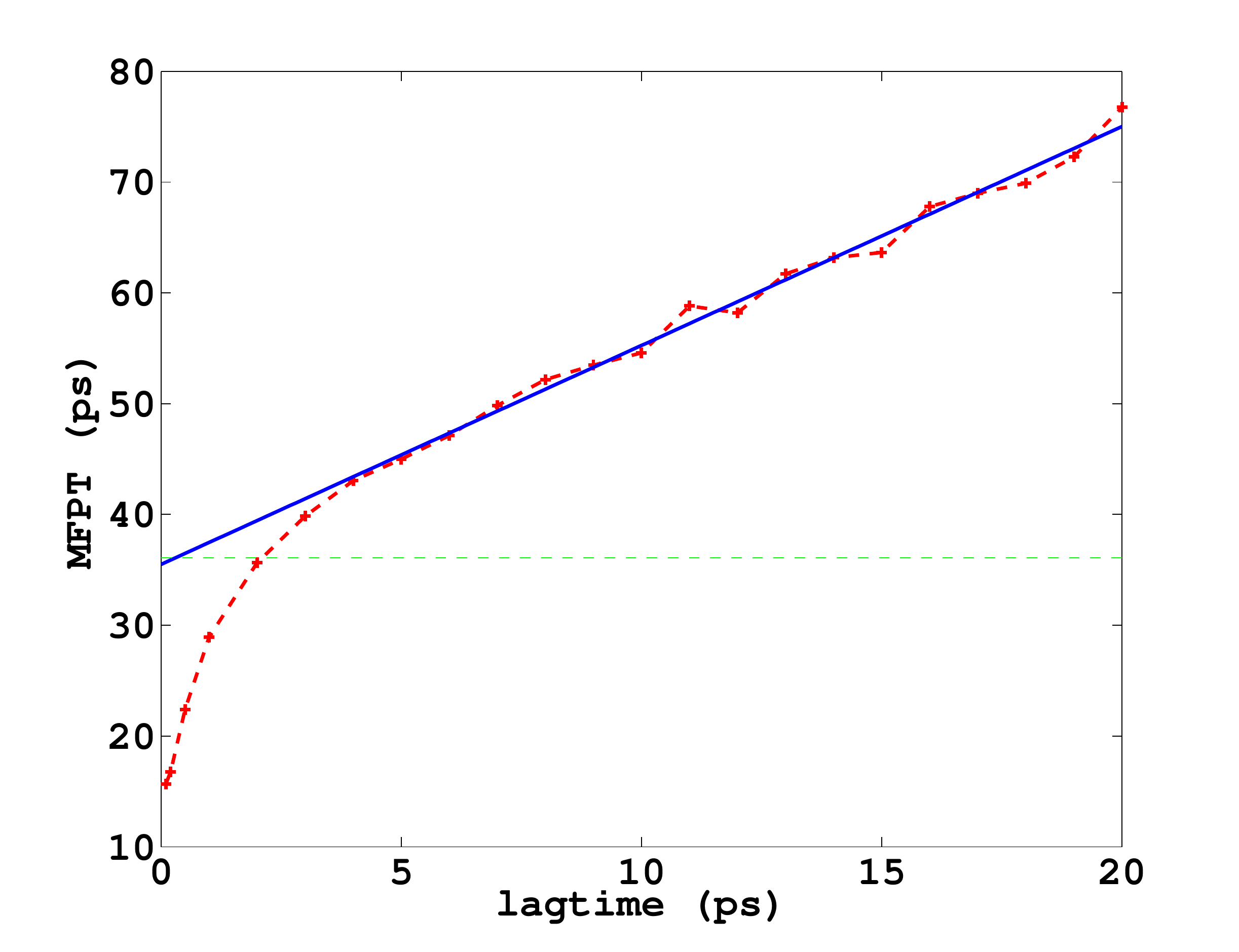}
 \caption{}
 \label{fig_num3potb}
 \end{subfigure}
 \\
 \begin{subfigure}[b]{0.48\textwidth}
 \includegraphics[width=0.98\textwidth]{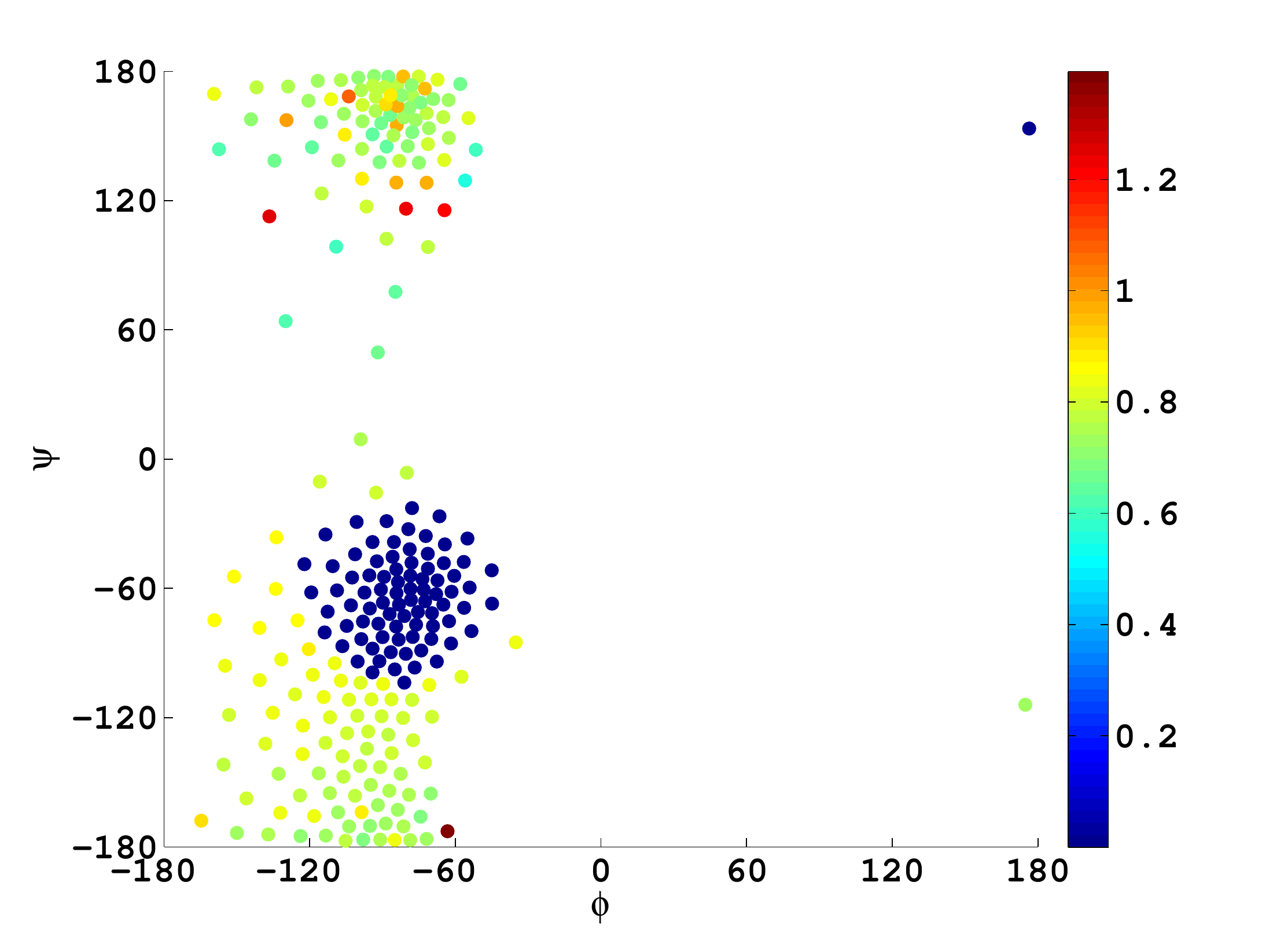}
 \caption{}
 \label{fig_num3potc}
 \end{subfigure}
 \begin{subfigure}[b]{0.48\textwidth}
 \includegraphics[width=0.98\textwidth]{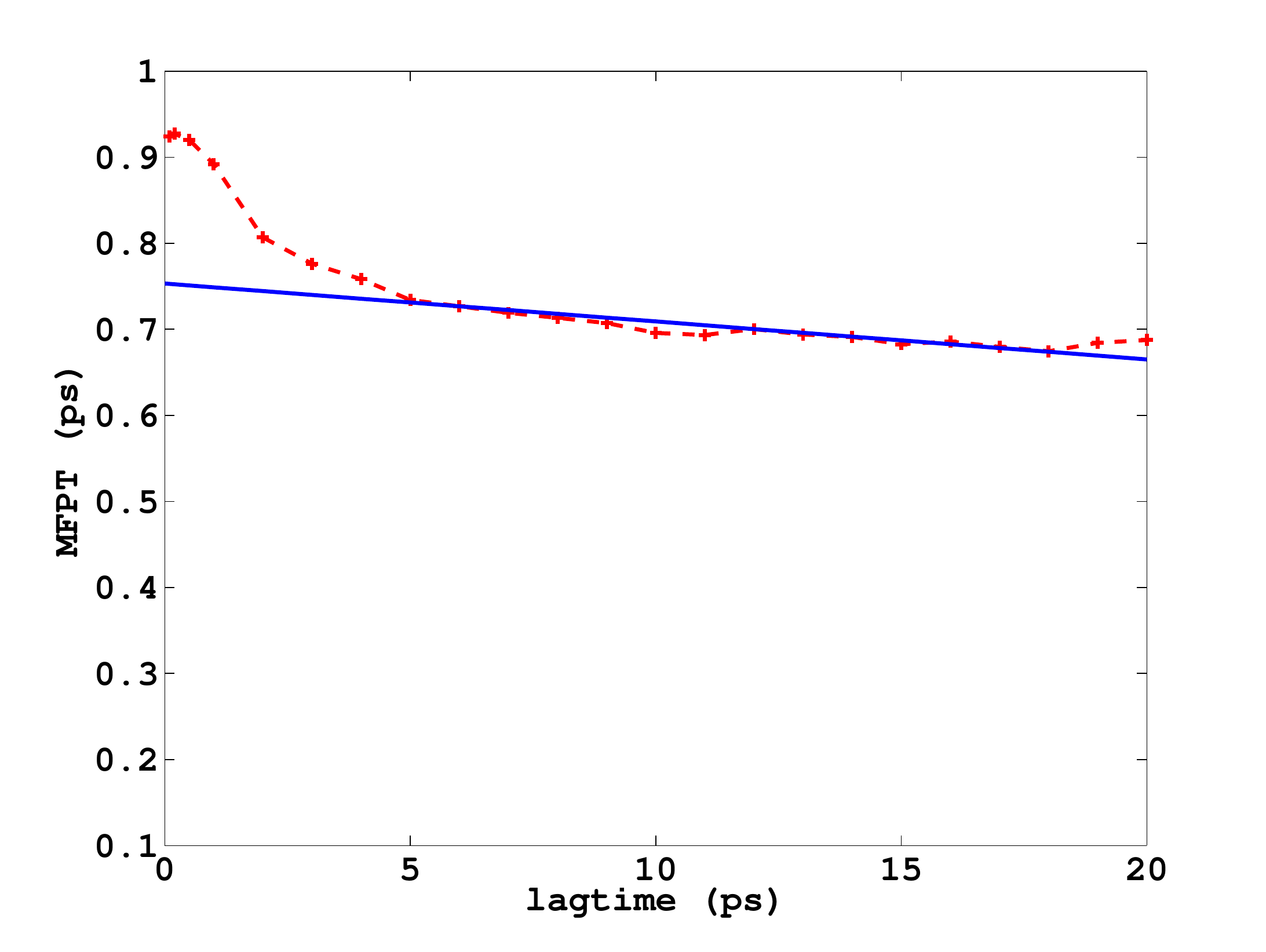}
 \caption{}
 \label{fig_num3potd}
 \end{subfigure}
 \\
 \begin{subfigure}[b]{0.48\textwidth}
 \includegraphics[width=0.98\textwidth]{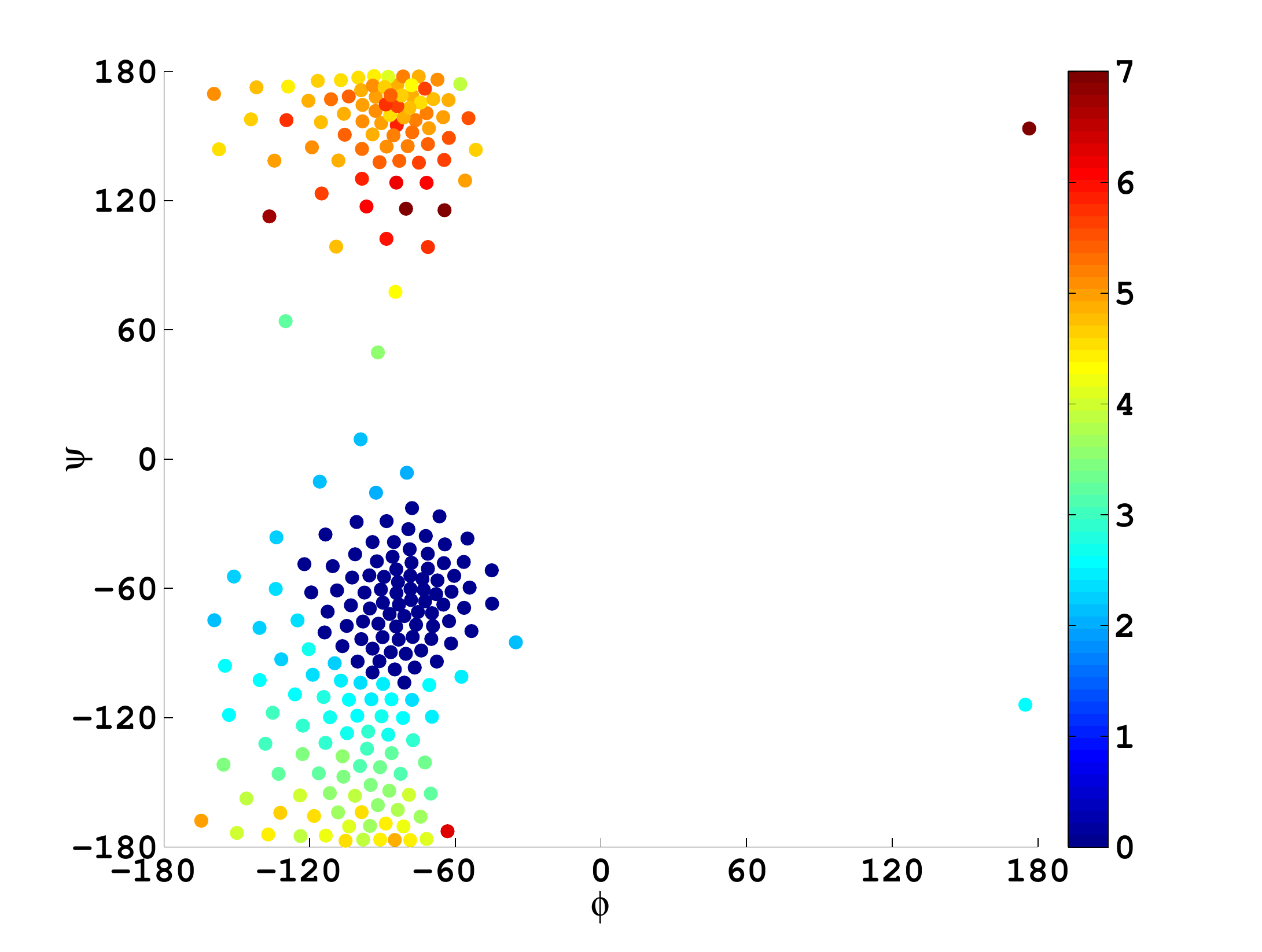}
 \caption{}
 \label{fig_num3pote}
 \end{subfigure}
 \begin{subfigure}[b]{0.48\textwidth}
 \includegraphics[width=0.98\textwidth]{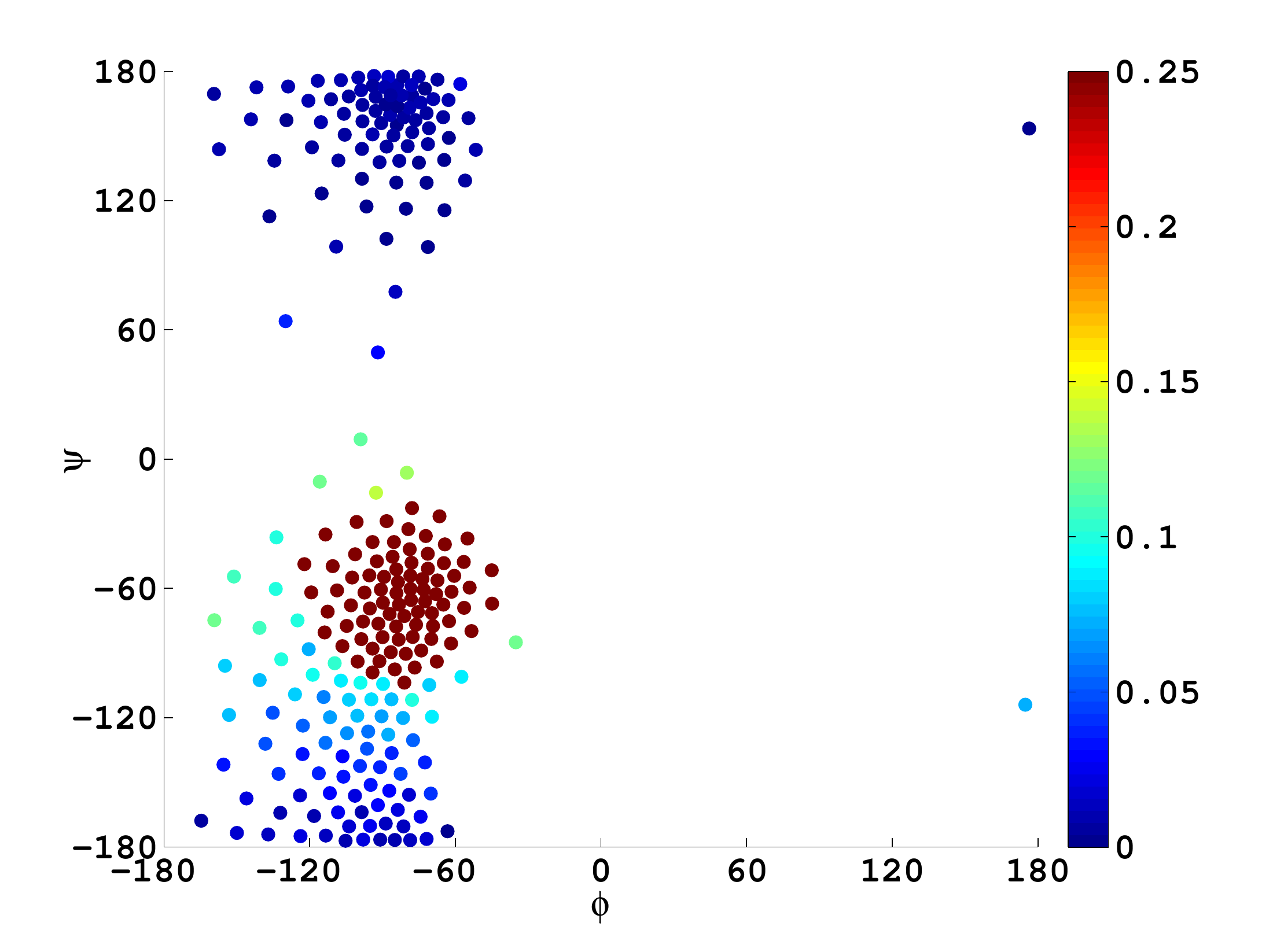}
 \caption{}
 \label{fig_num3potf}
 \end{subfigure}
 \label{fig_num3pot}
 \caption{Top: (\subref{fig_num3pota}) MFPT in ps from $\beta$ to $\alpha$ in $\phi$-$\psi$ space for $\tau=5ps$.(\subref{fig_num3potb}) MFPT as a function of $\tau$ (dashed line) and linear interpolation to $\tau=0$ (solid line). Green dashed line: Reference computed via slowest ITS. Middle: (\subref{fig_num3potc}) and (\subref{fig_num3potd}) same as (\subref{fig_num3pota}) and (\subref{fig_num1potb}), but for the controlled process. Bottom: Optimal cost (\subref{fig_num3pote}) and optimal strategy (\subref{fig_num3potf}) for the controlled process with lag time $\tau = 5ps$.}
\end{figure}

\FloatBarrier

\section{Conclusions}

We have developed a Galerkin projection method that leads to an approximation of certain optimal control problems for reversible diffusions by Markov decision problems. The approach is based on the dual formulation of the optimal control problem in terms of a linear boundary value problem that can be discretized in a straightforward way. In this article we propose a discretization that preserves reversibility and the generator form of the linear equations, i.e., the discretization of the infinitesimal generator of the original diffusion process can be interpreted as the infinitesimal generator of a reversible Markov jump process (MJP). The discretized linear boundary value problem admits again a dual formulation in terms of a Markov decision problem.  

Two special cases were discussed in detail: a Galerkin discretization based on a uniform box partition of state space and characteristic functions,  that was found to agree with the known Markov chain approximation to first order in the size of the boxes, and a sparse approximation that uses the basis of committor functions of metastable sets of the dynamics; the latter does not require that the metastable sets partition the state space, hence the method can be applied to high-dimensional problems as they appear, e.g., in molecular dynamics. The committor functions in this case need not be known explicitly, as it is possible to sample the generator matrices and the discrete cost functions by a Monte-Carlo method, similarly to what is done in the Markov state modelling approach to protein folding.  
We could prove an $L^{2}$ error estimate for the Galerkin scheme, moreover the discretization was shown to preserve basic structural elements of the continuous problem, such as duality, reversibility or properties of the invariant measure. Our numerical results showed very good performance of the incomplete partition discretization on a simple toy example and a high-dimensional molecular dynamics problem, even with only a few basis functions, which is in line with the theoretical error bounds presented in this paper. 

While we addressed the discretization error in this paper in great detail, we did not address the sampling error. In particular, for large systems our construction requires the coefficients of the MJP and therefore the transition rates between all metastable states as an input. This is not fully satisfactory. We believe that the optimal control framework presented here should be linked with Monte-Carlo methods for rare events, e.g., \cite{Hartmann2012,is_multiscale}, that exploit the same duality between optimal control and sampling to devise efficient importance sampling strategies as we did so as to reduce the sampling error. Moreover it would be desirable to use the MJP approach in a purely data-driven framework, e.g., for single molecule experiments or other optimal control applications in which a detailed microscopic model may not be feasible or may not be available. All this is ongoing work.

\section*{Acknowledgements}

The authors thank Marco Sarich and Christof Sch\"utte for helpful discussions and Frank No\'e for providing the Alanine dipeptide data. The research was funded by the DFG Research Centre Matheon. RB holds a scholarship of the Berlin Mathematical School.

\appendix

\section{Proof of Theorem \ref{Lemma_p}}\label{app_lemma_p}
 
Here we give the proof of Theorem \ref{Lemma_p} from Section \ref{sec_Galerkin}. For ease of notation, let $\|\cdot\| = \|\cdot\|_\mu$.

Let $\phi$ be the solution to (\ref{eq_bvp2}), and write $\phi = Q\phi + \phi_{\perp} = \phi_{||} + \phi_{\perp}$ with $\phi_\perp \in D^\perp$. The first step is to show that $\|\phi-\phi_{||}\| = \inf_{\psi\in D} \|\phi-\psi\|$, i.e. the infimum in the definition of $\bm{\varepsilon_0}$ is attained at $\phi_{||}$. But this is clear since for any $\psi\in D$, by orthogonality we have

$$\|\phi-\psi||^2 = \|\phi_{||} - \psi + \phi_{\perp}\|^2 = \|\phi_{||} - \psi\|^2 + \|\phi_{\perp}\|^2$$

which attains its minimum of $\bm{\varepsilon_0}^2 = \|\phi_{\perp}\|^2$ for $\psi = \phi_{||}$. By (\ref{eq_bvp2}), $\phi_{||}$ solves the equation

$$\cB(\phi,\psi) = \cB(\phi_{||},\psi) + \cB(\phi_{\perp},\psi) = 0\quad\forall \psi\in D,$$

and if we write $\phi_{||} = \sum_{i=1}^n \hat \phi_i^* \chi_i + 1\chi_{n+1}$ with $n$ unknown coefficients $\hat \phi_i^*$ (note that a general element of $D$ is of this form), this takes the matrix form

$$\hat B \hat \phi^* - c = F,$$

where in components we have $\hat B_{ij} = \cB(\chi_i,\chi_j)$, $c_i = -\cB(\phi_\perp, \chi_i) = -\langle \phi_\perp, B\chi_i\rangle_\mu$ and $F_i = -\langle \chi_i,B\chi_{n+1}\rangle_\mu$. On the other hand, the Galerkin solution $\hat\phi = \sum_i \hat \phi_i \chi_i$ satisfies $\hat B\hat \phi = F$ by \ref{eq_bvp2a}, hence we obtain

\begin{equation}
 \hat B (\hat \phi^* - \hat \phi) = c. \label{eq_matrix3}
\end{equation}

Now we can write

\begin{eqnarray*}
 \bm{\varepsilon}^2 & = & \|\phi_{||} + \phi_\perp - \hat\phi\|^2 = \|\phi_{||} - \hat\phi\|^2 + \|\phi_\perp\|^2\\
 & = & \left\langle \sum_i (\hat \phi_i^* - \hat \phi_i)\chi_i, \sum_j (\hat \phi_j^* - \hat \phi_j) \chi_j \right\rangle_\mu + \bm{\varepsilon_0}^2\\
 & = & (\hat \phi^* - \hat \phi)^T \hat M (\hat \phi^* - \hat \phi) + \bm{\varepsilon_0}^2
\end{eqnarray*}

where $\hat M_{ij} = \langle \chi_i, \chi_j\rangle_\mu$. The scalar product $\langle \cdot, \cdot\rangle_\mu$ on $D_0\subset V$ induces a natural scalar product on $\R^n$ by the isomorphism $\hat \phi \mapsto \sum_i \phi_i \hat \chi_i$:

$$\left\langle \sum_i \hat \phi_i \chi_i, \sum_j \hat \phi'_j \chi_j\right\rangle_\mu = \hat \phi^T \hat M\hat \phi' =: \langle \hat \phi, \hat \phi'\rangle_M$$

The error $\bm{\varepsilon}^2$ is exactly $\bm{\varepsilon_0}^2$ plus the distance between Galerkin solution and best approximation measured in this scalar product. There is also a natural bilinear form inherited from $\cB$ on $\R^n$:

$$\cB\left(\sum_i \hat \phi_i \chi_i,\sum_j \hat \phi_j' \chi_j\right) = \hat \phi^T\hat B \hat \phi' = \langle \hat \phi, \hat M^{-1} \hat B\hat \phi'\rangle_M$$

The Matrix $\hat M^{-1} \hat B$ is symmetric since $\cB(\cdot, \cdot)$ is symmetric. Moreover, since $\cB(\cdot, \cdot)$ is elliptic,

\begin{equation}
\langle \hat \phi, \hat M^{-1}\hat B \hat \phi\rangle_M = A\left(\sum_i \hat \phi_i \chi_i, \sum_j \hat \phi_j \chi_j\right) \geq \alpha_2 \left\langle \sum_i \hat \phi_i \chi_i, \sum_j \hat \phi_j \chi_j\right\rangle_\mu = \alpha_2 \langle \hat \phi,\hat \phi\rangle_M
\label{eq_elliptic}
\end{equation}

In particular, $\hat M^{-1} \hat B$ is positive, hence it has a positive and symmetric square root $\hat S^2 = \hat M^{-1}\hat B$. Now, for any $\hat \phi \in \R^n$ it holds by virtue of (\ref{eq_elliptic}),

\begin{eqnarray}
\langle \hat \phi,\hat \phi\rangle_M & \leq & \frac{1}{\alpha_2}\langle \hat \phi, \hat M^{-1}\hat B\hat \phi\rangle_M = \frac{1}{\alpha_2}\langle \hat S\hat \phi,\hat S\hat \phi\rangle_M \notag\\
&\leq& \frac{1}{\alpha_2^2} \langle \hat S\hat \phi, \hat M^{-1}\hat B \hat S \hat \phi\rangle_M = \frac{1}{\alpha_2^2}\langle \hat M^{-1}\hat B\hat \phi, \hat M^{-1} \hat B \hat \phi\rangle_M.
\label{eq_elliptic2}
\end{eqnarray}

Now we apply the inequality (\ref{eq_elliptic2}) to $\hat\phi^*-\hat\phi$ and use (\ref{eq_matrix3}):

\begin{equation}
\bm{\varepsilon}^2 \leq \bm{\varepsilon_0}^2 + \frac{1}{\alpha_2^2}\langle \hat M^{-1}c,\hat M^{-1}c\rangle_M.
\label{eq_bound}
\end{equation}

Now for some final simplifications, note that the orthogonal projection $Q$ onto $D_0$ can be written as

$$Q\psi = \sum_{i,j=1}^n \hat M^{-1}_{ij}\langle \chi_j, \psi\rangle_\mu \chi_i.$$

Using this we can write

\begin{eqnarray*}
\langle \hat M^{-1}c,\hat M^{-1}c\rangle_M & = & \sum_{ij} c_i\hat M^{-1}_{ij}c_j = \sum_{ij}\langle \chi_i, B\phi_\perp\rangle_\mu M^{-1}_{ij}\langle \chi_j, B\phi_\perp\rangle_\mu\\
& = & \left\langle\sum_{ij} M^{-1}_{ij} \langle \chi_j, B\phi_\perp\rangle_\mu \chi_i, B\phi_\perp\right\rangle_\mu = \langle QB\phi_\perp, B\phi_\perp\rangle_\mu\\
& = & \langle QB\phi_\perp, QB\phi_\perp\rangle_\mu
\end{eqnarray*}

To arrive at the final result, notice that

\begin{eqnarray*}
 \langle QB\phi_\perp, QB\phi_\perp\rangle_\mu & \leq & \left(\sup_{\phi'_\perp\in D^\perp}\frac{\langle QB\phi'_\perp, QB\phi'_\perp\rangle_\mu}{\langle \phi'_\perp,\phi'_\perp\rangle_\mu}\right) \cdot \langle \phi_\perp,\phi_\perp\rangle_\mu \\
 & = & \left(\sup_{\phi'_\perp\in D^\perp}\frac{\langle QBQ^\perp \phi'_\perp, QBQ^\perp \phi'_\perp\rangle_\mu}{\langle \phi'_\perp,\phi'_\perp\rangle_\mu}\right) \cdot \langle \phi_\perp,\phi_\perp\rangle_\mu \\
 & \leq & \left(\sup_{\phi'\in V}\frac{\langle QBQ^\perp \phi', QBQ^\perp \phi'\rangle_\mu}{\langle \phi',\phi'\rangle_\mu}\right) \cdot \langle \phi_\perp,\phi_\perp\rangle_\mu \\
 & = & \|QBQ^\perp\|^2 \langle \phi_\perp,\phi_\perp\rangle_\mu
\end{eqnarray*}

Plugging these inequalities into (\ref{eq_bound}) and dividing by $\bm{\varepsilon_0}^2$ completes the proof. $\Box$

\section{Best-approximation error bound}\label{app_best}

In this appendix, we prove lemma \ref{lemma_best}:

$$\bm{\varepsilon_0} = \|Q^\perp \phi\|_\mu \leq \|P^\perp \phi\|_\mu + \mu(T)^{1/2}\left[\kappa \|f\|_\infty + 2 \|P^\perp \phi\|_\infty \right].$$

Recall that $\kappa = \sup_{x\in T} \bE_x[\tau_{\cS\setminus T}]$ and $P$ is the orthogonal projection onto the subspace $V_c = \{v\in L^2(\cS,\mu), v=const \;\mbox{on every}\; C_i \} \subset L^2(\cS,\mu)$. Note that $P^\perp \phi = 0$ on $C$. The errors $\|P^\perp \phi\|$ and $\|P^\perp \phi\|_\infty$ measure how constant the solution $\phi$ is on the core sets. We write $\|\cdot\| = \|\cdot\|_\mu$ throughout the proof for convenience.

\begin{proof} The proof closely follows the proof of theorem (12) in \cite{Sarich_thesis}. The first step of the proof is to realize that the committor subspace $D$ where $Q$ projects onto can be written as $D = \{v\in L^2(\cS,\mu), v=const \;\mbox{on every}\; C_i, Lv = 0 \;\mbox{on}\; C\}$. To see this, note that the values $v$ takes on the $C_i$ can be used as boundary values for the Dirichlet problem $Lv = 0$ on $T$. A linear combination of committor functions is obviously a solution to this problem. But the solution to the Dirichlet problem must be unique, otherwise one can construct a contradiction to the uniqueness of the invariant distribution, see \cite{Sarich_thesis}.

By definition we have $\|Q^\perp \phi\| \leq \|\phi-I\phi\|$ for every interpolation $I\phi\in D$ of $\phi$. With the definition of $P$ from above, we will take $q=I\phi$ such that

\begin{equation}
Lq = 0 \;\mbox{on}\; T, \quad q = P\phi \;\mbox{on}\; \cS\setminus T.
\label{eq_qbvp}
\end{equation}

Now $D\subset V$, therefore $q\in V_c$ and $Pq = q$. Therefore (\ref{eq_qbvp}) is equivalent to

\begin{equation}
PLPq = 0 \;\mbox{on}\; T, \quad q = P\phi \;\mbox{on}\; \cS\setminus T.
\label{eq_qbvp2}
\end{equation}

Now define $e:=P\phi-q$. Then we have

$$PLPe = PLP(P\phi - q) = PLP\phi - PLPq = PL\phi - PLP^\perp \phi - PLPq$$

and by (\ref{eq_qbvp2}) and since $L\phi = f\phi$ on $\cS\setminus A \supset T$, we have

\begin{equation}
PLPe = Pf\phi - PLP^\perp \phi \;\mbox{on}\; T, \quad e = 0 \;\mbox{on}\; \cS\setminus T.
\label{eq_ebvp}
\end{equation}

Therefore, $e\in E_\Theta = \{v\in L^2(\cS,\mu), v=0 \;\mbox{on}\; \cS\setminus T\}$ and with $\Theta$ being the orthogonal projection onto $E_\Theta$, $e$ has to fulfil

$$\Theta PLP \Theta e = \Theta Pf\phi - \Theta PLP^\perp \phi.$$

Since $\Theta P = P\Theta = \Theta$, this can be written as

$$Re := \Theta L\Theta e = \Theta f \phi - \Theta LP^\perp \phi.$$

The operator $R=\Theta L\Theta$ is invertible on $E_\Theta$: If this wasn't the case, there would be a nontrivial solution $v$ to

$$ Lv = 0 \;\mbox{on}\; T, \quad v = 0 \;\mbox{on}\; \cS\setminus T.$$

But the solution to this boundary value problem is again unique, and hence there is only the trivial solution. This gives

\begin{equation}
e = R^{-1}\Theta f\phi - R^{-1}\Theta LP^\perp \phi,
\label{eq_e}
\end{equation}

and $\|R^{-1}\| = \frac{1}{|\lambda_0|}$ where $\lambda_0$ is the principal eigenvalue of $R$. Due to an estimate by Varadhan we have

$$\frac{1}{|\lambda_0|} \leq \sup_{x\in T}\bE_x[\tau_{\cS\setminus T}] =: \kappa,$$

see e.g. \cite{Bovier2009}. To complete the derivation we need to focus on the second term in (\ref{eq_e}). Since $R^{-1}$ is an operator on $E_\Theta$, we can write it as $R^{-1}\Theta LP^\perp \phi =: \Theta g$, where the function $\Theta g$ solves

$$\Theta L\Theta g = R\Theta g = \Theta L P^\perp \phi \Leftrightarrow \Theta L[\Theta g - P^\perp \phi] = 0 $$

by the definition of $R$ and $\Theta g$. Therefore $w:=\Theta g - P^\perp \phi$ solves the boundary value problem

\begin{equation}
Lw =  0 \;\mbox{on}\; T, \quad w = -P^\perp \phi \;\mbox{on}\; \cS\setminus T
\label{eq_wbvp}
\end{equation}

which implies that $\|w\|_\infty \leq \|P^\perp \phi\|_\infty$, this follows from Dynkin's formula or Lemma 3 in \cite{Sarich_thesis}. Finally,

$$\|\Theta g\| \leq \mu(T)^{1/2}\|\Theta g\|_\infty \leq  \mu(T)^{1/2}(\|P^\perp \phi\|_\infty + \|w\|_\infty) \leq 2\mu(T)^{1/2}\|P^\perp \phi\|_\infty$$

holds by the triangle inequality and the above considerations. Now focus on the first term in (\ref{eq_e}). Note that by the maximum principle, $\phi$ achieves its maximum of $1$ on the boundary of $A^c\supset T$, therefore $\max_{x\in T} |\phi(x)| \leq 1$. Then we have

$$\|\Theta f \phi\| \leq \mu(T)^{1/2}\|f\|_\infty \max_{x\in T}|\phi(x)| \leq \mu(T)^{1/2}\|f\|_\infty.$$

Now putting everything together, we arrive at

\begin{eqnarray*}
\|e\| &\leq & \|R^{-1}\|\|\Theta f \phi\| + \|R^{-1}\Theta L P^\perp \phi\| \\
& \leq &  \kappa \|\Theta f \phi\| + \|\Theta g\| \\
& \leq & \mu(T)^{1/2}\left[ \kappa \|f\|_\infty + 2\|P^\perp \phi\|_\infty \right].
\end{eqnarray*}

Finally, note that by the triangle inequality

$$\|Q^\perp \phi\| \leq \|\phi-q\| \leq \|\phi-P\phi\| + \|P\phi - q\| = \|P^\perp \phi\| + \|e\|$$

which completes the proof.
\end{proof}

\section{Finite-volume approximation}\label{app_FVA}

In this section we show (\ref{FVA_gen}), confirming that the Galerkin projection of $L$ onto step functions gives the finite volume approximation discussed in \cite{Latorre2011}. Recall the definitions of $S_{ij}$, $h_{ij}$ and $A_i$ given in figure (\ref{fig_fva_mesh}). We use the divergence representation $L\phi = e^{\beta V}\nabla\cdot(e^{-\beta V}\nabla \phi)$ with $\beta=\eps^{-1}$ and calculate

\begin{eqnarray*}
 \langle \chi_i,L\chi_j\rangle_\mu & = & \beta^{-1} \int_\mathcal{S} \:\chi_i \: e^{\beta V} \nabla \cdot \left(e^{-\beta V}\nabla \chi_j\right) e^{-\beta V}dx \\
 & = & \beta^{-1} \int_{A_i} \nabla \cdot \left(e^{-\beta V}\nabla \chi_j\right)dx \\
 & = & \beta^{-1} \int_{\partial A_i} e^{-\beta V}(\nabla \chi_j)\cdot \nu ds
\end{eqnarray*}

where $\nu$ is the surface normal vector field of $\partial A_i$. We write the integral over $\partial A_i$ as a sum over surface integrals over $S_{ij'}$ where $j'$ ranges over the set $\{i_l\}$ of neighbours of $i$ and approximate the surface integrals by a point evaluation of the integrand at the midpoint $\bar x_{ij'}$ times the area of $S_{ij'}$. That gives

$$ \langle \chi_i,L\chi_j\rangle_\mu \approx \beta^{-1} \sum_{j'\in \{i_l\}} m(S_{ij'}) e^{-\beta V(\bar x_{ij'})} \left( \nabla \chi_j \cdot \nu\right) |_{x = \bar x_{ij'}}$$

Now we can approximate the directional derivative of $\chi_j$ using a two-sided finite difference:

$$\nabla \chi_j \cdot \nu \Big|_{x = \bar x_{ij'}} = \frac{\nabla \chi_j \cdot h_{i,j'}}{m(h_{i,j'})} \Big|_{x = \bar x_{ij'}} \approx \frac{\chi_j(x_{j'}) - \chi_j(x_i)}{m(h_{ij'})} = \frac{\delta_{jj'}-0}{m(h_{ij'})}.$$

Hence in the sum over neighbours of $i$, only $j$ survives. Now we put everything together:

$$\langle \chi_i,L\chi_j\rangle_\mu \approx \beta^{-1} \frac{m(S_{ij})}{m(h_{ij})} e^{-\beta V(\bar x_{ij})}.$$

Finally, we divide by $\pi_i$ using (\ref{eq_FVA_dist}):

\begin{equation*}
 K_{ij} = \frac{1}{\pi_i}\langle \chi_i,L\chi_j\rangle_\mu \approx \frac{1}{\Delta_{ij}} e^{-\beta (V(\bar x_{ij})-V(x_i))}, \quad \Delta_{ij}^{-1} = \beta^{-1} \frac{m(S_{ij})}{m(h_{ij})m(A_i)}.
\label{FVA_gen2}
\end{equation*}

which confirms (\ref{FVA_gen}).

\section{Markov chain approximations}\label{app_MCA}

We now show (\ref{eq_mca1}). Let $i$ and $j$ be nearest neighbours, and $\beta = \eps^{-1}$. For a regular $d$-dimensional lattice with lattice spacing $h$,

$$\Delta_{ij} = \beta\frac{m(h_{ij})m(\Omega_i)}{m(S_{ij})} = \beta\frac{h h^{d}}{h^{(d-1)}} = \beta h^2.$$

Therefore, $G$ as given by (\ref{FVA_gen}) simplifies to

$$G_{ij} = \frac{1}{\beta h^2} e^{-\beta(V(\bar x_{ij}) - V(x_i))}.$$

We introduce the function $\hat W_v(i) = -\beta^{-1}\log v(i)$. Then, for neighbours $i,j$,

$$G^v_{ij} = \frac{1}{\beta h^2} e^{-\beta(V(\bar x_{ij}) - V(x_i) + \hat W_v(j) - \hat W_v(i)) }.$$

Now we specialise to the one-dimensional case, thus $j = i\pm 1$. We write $V(\bar x_{i,i \pm 1}) -V(x_i) = \pm \frac{h}{2} \nabla V(x_i) + \mathcal{O}(h^2)$. Expanding the exponential gives

\begin{eqnarray*}
 G^v_{i,i \pm 1} & = & \frac{1}{\beta h^2}\left(1 - \frac{\beta h}{2}\left(\pm \nabla V(x_i)\right) - \beta(\hat W_v(i \pm 1) - \hat W_v(i)) + \mathcal{O}(h^2)\right)\\
 & = & \frac{1}{h^2}\left(\beta^{-1} - \frac{h}{2}\left(\pm \nabla V(x_i) +2\frac{\hat W_v(i \pm 1) - \hat W_v(i)}{h}\right) + \mathcal{O}(h^2)\right)\\
 & = & \frac{1}{h^2}\left(\beta^{-1} \mp \frac{h}{2} \left(\nabla V(x_i) -\alpha_v^\pm(i) \right) + \mathcal{O}(h^2)\right)
\end{eqnarray*}

with the definition

$$\alpha_v^\pm(i) := \pm \left(-2\frac{\hat W_v(i\pm 1)-\hat W_v(i)}{h}\right).$$

Now consider the difference between $\alpha^+_v(i)$ and $\alpha^-_v(i)$:

$$\alpha^+_v(i) - \alpha^-_v(i) = -2h\frac{\hat W_v(i+1)-2\hat W_v(i) + \hat W_v(i-1)}{h^2}$$

Assuming that $\hat W_v$ converges to a twice differentiable function, the difference between $\alpha^+_v(i)$ and $\alpha^-_v(i)$ is of order $h$. In other words we may write $\alpha^\pm_v(i) = \alpha_v(i) + \mathcal{O}(h)$ where

$$\alpha_v(i) = \frac{1}{2}(\alpha^+_v(i) + \alpha^-_v(i)) =  \frac{1}{h\beta}\left(\log v(i+1) - \log v(i-1)\right).$$

Then

$$G^v_{i,i\pm 1} = \frac{1}{h^2}\left(\beta^{-1} \mp \frac{h}{2} \left(\nabla V(x_i) -\alpha_v(i) \right) + \mathcal{O}(h^2)\right)$$

which confirms (\ref{eq_mca1}). Now we show (\ref{eq_mca_cost}). We use the representation (\ref{kv}) for $k^v(i)$:

\begin{eqnarray*}
 \beta k^v(i) & = &\sum_{j\neq i} G_{ij}\left\{\frac{v(j)}{v(i)} \left[\log \frac{v(j)}{v(i)} - 1\right] + 1\right\} \\
 & = & G_{i,i+1}\frac{v(i+1)}{v(i)}\left(\log \frac{v(i+1)}{v(i)} - 1\right) + G_{i,i-1}\frac{v(i-1)}{v(i)}\left(\log \frac{v(i-1)}{v(i)} - 1\right) - G_{ii}
\end{eqnarray*}

Now we write this in terms of the shorthands $\alpha^\pm := \pm h^{-1} \log\frac{v(i\pm 1)}{v(i)}$. Notice that $\alpha^\pm = \frac{\beta}{2}\alpha^\pm_v(i)$ and use the formula for $G$ above:

\begin{eqnarray*}
 \beta k^v(i) & = & \frac{1}{\beta h^2}\left[e^{h\alpha^+} (h\alpha^+ - 1) - e^{-h\alpha^-}(h\alpha^- + 1)+2\right]
   \\&& - \frac{\nabla V(x_i)}{2h}\left[e^{h\alpha^+} (h\alpha^+ - 1) + e^{-h\alpha^-}(h\alpha^- + 1)\right] + \mathcal{O}(h) \\
 & = & \frac{1}{\beta h^2}\left[-1-1+2\right] + \frac{1}{\beta h}\left[\alpha^+ - \alpha^- - \alpha^+ + \alpha^-\right] + \frac{1}{2\beta}\left[(\alpha^+)^2+(\alpha^-)^2\right] \\
 & & - \frac{\nabla V(x_i)}{2h}[-1+1] - \frac{\nabla V(i)}{2}\left[\alpha^+ -\alpha^+ + \alpha^- - \alpha^-\right] + \mathcal{O}(h) \\
 & = & \frac{1}{2\beta}\left[(\alpha^+)^2+(\alpha^-)^2\right] + \mathcal{O}(h)\\
 & = & \frac{\beta}{4}\alpha^2_v(i) + \mathcal{O}(h) = \frac{\beta}{4}\alpha^2_v(i) + \mathcal{O}(h).
\end{eqnarray*}

This confirms (\ref{kv}). In the second step, we have used Taylor expansions of $e^{h\alpha^\pm}$ up to second order. In the last step, we have used $\alpha_v^\pm(i) = \alpha_v(i) + \mathcal{O}(h)$.

\section{Sampling of the discretized running cost}\label{app_samplingF}

We show the sampling formula (\ref{eq_samplingF}) for $F$:

$$F_{ij} = \bE_\mu\left[f(X_t) \chi_{\{\tilde X_t^{+}= j\}}\middle| \tilde X_t^- = i\right].$$

Recall that since the dynamics is reversible, $\chi_i(x) = \bP(\tilde X_t^\pm = i|X_t = x)$ with $\tilde X_t^\pm$ being the forward and backward milestoning processes defined in Section \ref{sec_interpretation}. Then,

\begin{eqnarray*}
 \hat F_{ij} & = & \int f(x)\chi_j(x) \chi_i(x)\mu(x)dx = \int f(x)\bP(\tilde X_t^{+} = j|X_t = x)\bP(\tilde X_t^{-} = i,X_t = x)dx \\
 & = & \int f(x)\bP(\tilde X_t^{+} = j,\tilde X_t^{-} = i, X_t = x)dx \\
 & = & \int f(x)\bP(\tilde X_t^{+} = j,\tilde X_t^{-} = i | X_t = x) \bP(X_t = x)dx \\
 & = & \int f(x)\bE\left(\chi_{\{\tilde X_t^{+} = j,\tilde X_t^{-} = i\}} \middle| X_t = x\right) \bP(X_t = x)dx \\
 & = & \bE_\mu\left[f(X_t) \chi_{\{\tilde X_t^{+}=j,\tilde X_t^{-} = i\}}\right].
\end{eqnarray*}

This completes the proof.

\bibliographystyle{alpha}

\bibliography{biblio2,control}

\end{document}